\documentclass[11pt]{article}

\usepackage{amsfonts}

\usepackage{amssymb,latexsym}

\usepackage{amsmath,amscd}

\usepackage{theorem}

\usepackage[pdftitle={Paper}, pdfauthor={Alejandro Cabrera}]{hyperref}

\newtheorem{lemma} {Lemma} [section]

\newtheorem{proposition} [lemma] {Proposition}

\newtheorem{theorem} [lemma] {Theorem}

\newtheorem{corollary} [lemma] {Corollary}

\newtheorem{definition}[lemma] {Definition}

\theorembodyfont{\rm}

\newtheorem{example}[lemma] {Example}

\newtheorem{remark}[lemma]{Remark}

\newenvironment{proof}{{\sc Proof:}}{{\hspace*{\fill} $\square$\\}}

\newcommand{\Cour}[1]      {[\![#1]\!]}

\def\R{\mathbb{R}}

\def\gg{{\mathfrak{g}}}

\def\om{\omega}
\def\tom{\tilde{\omega}}

\def\Ker{\mathrm{Ker}}

\def\tM{\tilde{M}}
\def\tE{\tilde{E}}
\def\tA{\tilde{A}}
\def\tmu{\tilde{\mu}}
\def\tU{\tilde{U}}

\def\F{\mathcal{F}}

\def\rra{\rightrightarrows}
\def\a{\alpha}
\def\b{\beta}

\def\uom{\underline{\om}}

\def\bna{\bar{\nabla}}
\def\mus{\mu^\sharp}

\def\tmu{\tilde{\mu}}
\def\teta{\tilde{\eta}}
\def\te{\tilde{e}}

\def\tG{\tilde{G}}

\def\t{\tau}
\def\inv{inv}
\def\tv{\tilde{v}}

\def\tchi{\tilde{\chi}}
\def\tpi{\tilde{\pi}}
\def\trho{\tilde{\rho}}

\def\K{\mathcal{K}}
\def\q{\mathfrak{q}}
\def\gom{\om}

\begin{document}

\title{Quotients of multiplicative forms and Poisson reduction}

\author{Alejandro Cabrera\thanks{{\tt alejandro@matematica.ufrj.br}, {Instituto de Matem\'atica, Universidade Federal de Rio de Janeiro, Av. Athos da Silveira Ramos 149, Cidade Universit\'aria, Ilha do Fund\~ao, 21941-909 Rio de Janeiro - Brasil. }}, Cristian Ortiz\thanks{\texttt{cortiz@ime.usp.br}, {Instituto de Matem\'atica e Estat\'istica, Universidade de S\~ao Paulo, Rua do Mat\~ao 1010, Cidade Universit\'aria, 05508-090 S\~ao Paulo - Brasil.}}}

\maketitle

\begin{abstract}
In this paper we study quotients of Lie algebroids and groupoids endowed with compatible differential forms. We identify Lie theoretic conditions under which such forms become basic and characterize the induced forms on the quotients. We apply these results to describe generalized quotient and reduction processes for (twisted) Poisson and Dirac structures, as well as to their integration by (twisted, pre-)symplectic groupoids. In particular, we recover and generalize several known results concerning Poisson reduction.
\end{abstract}

\tableofcontents


\section{Introduction} 

In this paper we study Lie-theoretic aspects of basic multiplicative forms \cite{bcwz, CDW} on Lie groupoids and consider applications to Poisson, Dirac and other geometries. More precisely, we will consider the set of data $(G,\om,\q)$ consisting of a Lie groupoid $G\rra M$, $\om\in \Omega^k(G)$ a differential form and $\q: G\to \tG$ a Lie groupoid morphism, detail compatibility conditions among them and discuss the contents of the paper using this notation.

\medskip

{\bf Lie theory underlying each ingredient in $(G,\om,\q)$.} In parallel to standard Lie theory for Lie groups and Lie algebras, a Lie groupoid $G$ has an associated infinitesimal analogue consisting of a Lie algebroid $A=Lie(G)$. We will assume that the reader is familiar with the basic elements of this theory. We recall that the main difference between the group and groupoid Lie theory is that not every algebroid is the differentiation of a groupoid, i.e. there is no extension of Lie's third theorem for groupoids. Nevertheless, when $A=Lie(G)$, in which case we say that $A$ is integrable, the analogues of Lie's first two theorems hold. Integrable Lie algebroids were fully characterized in \cite{crainicfernandes}. In particular, differentiation of maps at identities induces a functor
$$ \q \mapsto (Lie(\q):Lie(G) \to Lie(\tG))$$
from the category of Lie groupoids to the category of Lie algebroids. A morphism $\q$ is completely determined by $Lie(\q)$ when $G$ has simply connected source-fibers. Moving towards the differential form $\om$, we observe that when $k=0$ so that $\om:G\to \R$ is a smooth function, a natural compatibility condition between $\om$ and the Lie groupoid structure on $G$ is 
$$ \om(gh)=\om(g)+\om(h), \text{ for all composable $g,h\in G$,}$$
thus defining a groupoid $1$-cocycle. This condition can be naturally generalized to higher differential forms $\om\in \Omega^k(G)$ defining so-called multiplicative forms on $G\rra M$. These multiplicative forms can be differentiated yielding infinitesimally multiplicative forms on $A$ or, equivalently, $A$-morphic linear forms which we denote $Lie(\om)$ \cite{bcwz,bursztyncabrera}. The underlying Lie theory for such ``morphic geometric structures" has been intensively studied \cite{MX, MX2, bco, bursztyncabrera} mainly motivated by its applications in Poisson geometry, which we now recall.

\medskip

{\bf Motivation and the link to Poisson and related geometries.}
Morphic structures are well-known to be interesting tools to study underlying geometric structures. For example, a Lie algebroid endowed with a morphic $2$-form which is non-degenerate is equivalent to a Poisson structure on its base \cite{CDW}. If the non-degeneracy condition is relaxed in a particular controlled way, the induced structure on the base is a Dirac structure \cite{bcwz}. Thinking in Lie theoretic terms, such objects are integrated by Lie groupoids endowed with multiplicative $2$-forms satisfying corresponding non-degeneracy conditions. A paradigmatic case is that of a symplectic groupoid \cite{CDW, CF, MX2}, which provides a Lie theoretic integration of the underlying Poisson structure on its manifold of objects.

On the other hand, symmetries and conserved quantities lead to reduction and quotient procedures for Poisson and Dirac structures. Many instances of these have been studied in the literature, with increasing degree of generality as to what the initial reduction data can be. See \cite{CZ12} for a setting particularly relevant to this paper. Thinking of Poisson and Dirac structures in terms of morphic forms, in this paper we provide a general quotient scheme in which the initial reduction data can be very general and encompasses most known procedures. By construction, the outcome of the reduction procedure is a Poisson or Dirac structure and the formalism allows us to address the question of how to provide a Lie theoretic integration of these quotients. The key technique is to consider quotients of multiplicative forms defined on the Lie groupoids behind the quotient data. Another interesting feature of our formalism is that it allows us to extend the results to both the cases of \emph{twisted} and \emph{higher} (where $k\geq 2$) Poisson and Dirac structures \cite{BMR} in a natural, almost effortless, way.

\medskip

{\bf Basic forms and the role of special simple quotients.} Following the above motivations, in the data $(G,\om,\q)$ we will be assuming that $\om$ is $\q$-basic so that there is an induced $k$-form $\tom\in \Omega^k(\tG)$ such that $\om=\q^*\tom$. Moreover, we restrict our attention to cases in which $\q$ defines a {\bf simple quotient} of the Lie groupoid $G\rra M$. By this nomenclature, we mean that $\q$ defines a fibration of Lie groupoids, in the sense of K. Mackenzie \cite[\S 2.4]{mackenziebook}, and that the usual $\q$-fibers are connected. The reason for the second condition is to have the standard infinitesimal characterization of $\om$ being $\q$-basic as being horizontal and infinitesimally invariant. The condition of $\q$ being a fibration of Lie groupoids is easier to understand at the infinitesimal level (see \cite[\S 4.4]{mackenziebook} and \cite{higgins-mackenzie}): $Lie(\q)$ must be fiberwise surjective and the Lie algebroid structure on $Lie(\tG)$ is obtained as a quotient of the one defining $Lie(G)$. The compatibility condition between $Lie(\om)$ and $Lie(G)$, saying that the form is morphic for the algebroid structure, is then automatically transported to the quotient structures. Similarly, following \cite{mackenziebook}, we can regard a fibration $\q:G\to \tG$ as a quotient map where the structure in $\tG$ can be obtained from the one on $G$ by a quotient procedure and the compatibility condition saying that $\tom$ is multiplicative in $\tG$ also follows from this perspective. 
The nomenclature ``simple quotient" comes from thinking of any smooth surjective submersion $X\to \tilde X$ with connected fibers as the quotient map onto the leaf space of a simple foliation on $X$. We will also provide a detailed study of the particular case in which the quotient $\q$ is defined by a foliation tangent to the kernel of $\om$,
$$ X \in Ker(\om) \iff i_X \om = 0.$$
For the distribution $Ker(\om)$ to be integrable and for the underlying foliation to define a simple quotient, additional conditions need to be required, and we summarize them by saying that $\om$ is {\bf kernel-reducible}. Observe that the underlying quotient form $\tom$ will have trivial kernel in this case, which makes it, in particular, an interesting source for non-degenerate morphic forms as in the Poisson-geometric motivation.

\medskip

{\bf Main results and outline.} We are now ready to provide a summary of our main results. These concern data $(G,\om,\q)$ as above in which $\om$ is a multiplicative $k$-form on $G$ and $\q$ defines a simple quotient of $G\rra M$ such that $\om$ is $\q$-basic. The infinitesimal analogue is given by data $(A,\om_A,q)$ where $A$ is a Lie algebroid, $\om_A$ is a linear $A$-morphic $k$-form and $q$ defines a simple quotient of $A$ such that $\om_A$ is $q$-basic.

In Section \ref{subsec:general} below, we recall elementary facts about basic differential forms on manifolds, introduce the general nomenclature to be used and highlight the case in which the underlying quotient is defined by the kernel of the form.

In Section \ref{sec:vbalg}, we provide the first set of results involving the characterization of conditions for the infinitesimal data $(A,\om_A,q)$ in terms of underlying simple pieces which we refer to as \emph{components}. In particular, we characterize when a given simple quotient $q$ of vector bundles is such that both $A$ and $\om_A$ are $q$-basic inducing such structures on the quotient and provide formulas in terms of the original data. We also study the case in which $\om_A$ is kernel-reducible and provide a characterization of the underlying conditions in terms of components.

In Section \ref{sec:mult}, we study the data $(G,\om,\q)$ focusing on providing infinitesimal characterization of the underlying compatibility conditions in terms of $(Lie(G),Lie(\om),Lie(\q))$. In particular, we show that if $G$ has connected source fibers and $Lie(\om)$ is $Lie(\q)$-basic, then $\om$ is $\q$-basic and we provide formulas characterizing the quotient form $\tom$. We also include a detailed study of kernel-reducible multiplicative forms. In particular, we provide an infinitesimal criterion for $Ker(\om)$ having constant rank and being involutive as a distribution. We think that some of the underlying technical results could be of independent interest.

In Section \ref{sec:Applications}, we apply the general results to describe generlized quotient procedures for Poisson and Dirac structures. We regard data of  the form $(A,\om_A,q)$ for which the quotient $\tom_A$ satisfies the appropriate non-degeneracy conditions as {\bf Poisson (resp. Dirac) quotient data} and describe the underlying quotient geometry. We also apply the Lie theoretic results to provide integration schemes for these quotient structures. Moreover, we show how several results in the literature involving quotient or reduction procedures can be seen as particular cases of our constructions and introduce generalizations to both twisted and higher Poisson and Dirac structures.

\medskip

{\bf Comments about integrability of simple quotients.} The integrability problem for simple quotient maps is non-trivial and is not addressed in this paper. Namely, given a simple quotient $q$ of a Lie algebroid $A$, the integration problem consists in characterizing when there are integrations $G$ of $A$ and $\q$ of $q$ such that $\q$ defines a simple quotient of $G\rra M$. Notice that, along the lines of this paper, we would also need to include the additional requirement that $G$ supports an integration $\om$ of a given $\om_A$. Several particular cases of this integration problem appear in the literature and, from these, is clear the the source-simply-connected integration does not provide a solution in general and adhoc techniques have to be considered. A particular case in which the quotient $q$ is of ``pullback type" (see Remark \ref{rmks:pullbackq}) and $k=2$ can be treated by the techniques introduced by D. \'Alvarez \cite{Alv} and the outcome is a complete characterization of the underlying integrability conditions (see Example \ref{ex:intqpullbacktype}). For general $q$, as far as the authors know, the integrability problem has to be studied in a case-by-case basis and a similar discussion applies to the quotient map behind $Ker(\om)$. In this paper, we assume that integrations $\q$ and $\om$ can be found and focus on the study of the compatibility conditions between $\q$ and $\om$. In the case of kernel-reduction, we assume that, once $Ker(\om)\subset TG$ is established to define a regular integrable distribution, the underlying foliation defines a simple quotient into its leaf space (leaf spaces of multiplicative foliations were studied in general by M. Jotz in \cite{J}).

\bigskip

{\bf Acknowledgments:} A.C. would like to thank Henrique Bursztyn and Rui Loja-Fernandes for conversations about the contents of this paper. AC was supported by CNPq grants 305850/2018-0 and 429879/2018-0 and by FAPERJ grant JCNE  E­26/203.262/201. C.O. was supported by grants 2016/01630-6 and 2019/14434-9 Sao Paulo Research Foundation. C.O. would like to thank the Institute of Mathematics of the University of Antioquia - Medellin for the hospitality while part of this work was carried out. In particular, C. O. thanks Pedro Hernandez Rizzo and Omar Saldarriaga for all the support which made possible a research visit to Medellin during the second half of 2019.

\subsection*{Simple quotients and basic differential forms} \label{subsec:general}
Let $X$ be a manifold. Given a regular foliation $\F$ on $X$, we denote by $T\F \subseteq TX$ the integrable regular distribution given by the tangent spaces to the leaves of $\F$. We recall that a foliation is {\bf simple} if the leaf space $X/\F$ admits a smooth structure such that the quotient map 
$$ q: X \to X/\F$$
defines a surjective submersion. In this case, the smooth structure on the quotient is unique. 

Conversely, given a surjective submersion 
$$ q: X \to \tilde X$$
with connected fibers, there is an induced simple foliation $\F:=\F(q)$ given by the $q$-fibers such that $\tilde X = X/\F(q)$. In this case, we say that $q$ defines a {\bf simple quotient of $X$}. The information behind a simple quotient of $X$ is encoded uniquely by the underlying foliation $\F$ or, equivalently, by the tangent distribution $T\F=Ker(Tq)$.

A differential form $\omega \in \Omega^k(X)$ is {\bf $q$-basic} if there exists $\tilde \omega\in \Omega^k(\tilde X)$ such that
$$ \omega = q^*\tilde \omega.$$

In this case, the $k$-form $\tilde \omega\in \Omega^{k}(\tilde X)$ is uniquely defined by the previous equation. We will sometimes refer to it as 
$$ \tilde \omega = q_* \omega.$$

Since the fibers of a simple quotient $q$ are connected, a $k$-form $\omega$ is $q$-basic if and only if
\begin{enumerate}
	\item $\omega$ is {\bf $q$-horizontal}: $i_V\omega =0$ for all $V\in Ker(Tq)$,
	\item $\omega$ is {\bf infinitesimally $q$-invariant}: $L_V\omega =0$ for all $V\in \mathfrak{X}(X)$ with $\ Tq(V)=0$.
\end{enumerate}
A particularly interesting situation arises when $\om\in \Omega^k(X)$ has {\bf constant rank}, meaning that 
$$ x\mapsto \mathrm{dim}(Ker(\om|_x)),$$
does not depend on $x\in X$, where $Ker(\om|_x):= \{ V\in T_xX: i_V \om|_x = 0 \}$.

In this case, assuming that 
$$Ker(\om|_x) \subset Ker(d\om|_x) \ \forall x\in X,$$
a direct application of the Cartan formula implies that $Ker(\om)\subset TX$ defines a regular involutive distribution on $X$. We denote by $\F_{Ker(\om)}$ the induced foliation on $X$. Assuming $\F_{Ker(\om)}$ is simple, defining a simple quotient $q:X\to \tilde X=X/\F_K$ onto the leaf space, we get that $\om$ is automatically $q$-basic, hence it induces a $k$-form $q_*\om$ on $\tilde X$ which has trivial kernel. In this case, we say that $\om$ is {\bf kernel-reducible}.

\medskip

In this paper, we consider the above simple picture in the cases in which $X=E$ is the total space of a Lie algebroid and $X=G$ is the space of arrows of a Lie groupoid. In both situations we assume that $\omega$ and $q$ are suitably compatible with the underlying structure on $X$.  


\section{Quotients of morphic forms on vector bundles and Lie algebroids}\label{sec:vbalg}

\subsection{Preliminaries: simple quotients of vector bundles and Lie algebroids}\label{subsec:simpleqvec}

In this subsection, we consider simple quotients in the categories of vector bundles and of Lie algebroids. We describe the data behind such quotients in terms of ``components" which are simpler to deal with. The general constructions underlying this subsection have been developed in \cite{mackenziebook} and we extract from that reference the key concepts and results that we will use in the main body of the paper.

\subsubsection*{Simple quotients of vector bundles and their components}

We begin by defining simple quotients of vector bundles.
\begin{definition}
A {\bf simple quotient of a vector bundle} $E\to M$ is a vector bundle morphism
$$ (q,q_M): (E\to M) \to (\tilde E\to \tilde M)$$
such that $q_M:M\to \tilde M$ is a surjective submersion with connected fibers and $q:E\to \tilde E$ is fiberwise surjective. 
\end{definition}
Notice that $q_M$ can be recovered from $q$ by restriction to the zero section. Also, if we think of the total spaces as manifolds, then $q$ defines a surjective submersion with connected fibers. We sometimes refer to the quotient bundle as to
$$ \tilde E = q_* E.$$

The discussion below follows \cite[\S 2.1]{mackenziebook}, and we review the main results to fix notations and ideas. The simple foliation $\F(q)$ of $E$ given by the $q$-fibers admits a simplified description in terms of underlying "components", which we first axiomatize. 

\begin{definition} 
We say that a triple $(q_M,K,\Delta)$ defines {\bf quotient data} for $E\to M$ if:
\begin{enumerate}
	\item $q_M:M \to \tilde M$ is a surjective submersion with connected fibers onto a manifold $\tilde M$,
	\item $K\subseteq E$ defines a vector subbundle over $M$,
	\item $\Delta$ defines a smooth assignment taking a pair of points $x,y\in M$ on the same $q_M$-fiber, $q_M(x)=q_M(y)$, to a linear isomorphism 
	$$ \Delta_{x,y}: E_y/K_y \overset{\sim}{\to} E_x/K_x, $$
	satisfying
	$$ \Delta_{x,x} = id_{E_x/K_x}, \ \Delta_{x,y}\circ \Delta_{y,z} = \Delta_{x,z}.$$
\end{enumerate}
\end{definition}
Notice that the $\Delta$ above can be described as a representation of the submersion groupoid $M\times_{\F(q_M)} M \rightrightarrows M$  associated to the foliation $\F(q_M)$ on the fiberwise quotient vector bundle $E/K \to M$. See \cite[Def. 2.1.5]{mackenziebook}, where the pair $(K,\Delta)$ is called a \emph{subbundle system in $E$}.

\begin{remark}\label{rmk:integrablebna}
	The representation $\Delta$ can be differentiated to a partial connection 
	\begin{eqnarray*}\bna&: &\Gamma(Ker(Tq_M)) \times \Gamma(E/K) \to \Gamma(E/K),\\ 
		 \bna_v(\bar e)&:=& \frac{d}{d\epsilon}|_{\epsilon =0} \Delta_{x,y(\epsilon)} \bar e(y(\epsilon))\text{, }v=\frac{d}{d\epsilon}|_{\epsilon =0}y(\epsilon) 
	\end{eqnarray*}
	which is flat. In this case, we say that $\bna$ is {\bf integrable} to a representation of $M\times_{\F(q_M)} M$. Notice that the data $(q_M,K,\bna)$ define the distribution $T\F(q) \subset TE$.
	Conversely, given $(q_M,K,\bna)$ with $\bna$ a flat connection as above, they uniquely determine an integrable regular distribution $T\F_{(q_M,K,\bna)} \subset TE$ which is \textbf{linear}, meaning that it defines a sub-double vector bundle of $TE$, see \cite{mackenziebook}). The underlying quotient $E\to E/\F_{(q_M,K,\bna)}$ defines a simple quotient of $E\to M$ iff $\bna$ integrates to a representation of $M\times_{\F(q_M)} M$.
\end{remark}
Given a simple quotient $q$ of $E\to M$, we define its {\bf components} to be $(q_M,K,\Delta)$ where
\begin{equation}\label{eq:qcomponents} K_x : = Ker(q|_{E_x}), \ \Delta_{x,y} =  (q|_{E_x} mod \ K_x)^{-1} \circ (q|_{E_y} mod \ K_y).\end{equation}
It is easy to verify that the components of $q$ define quotient data for $E\to M$ and that, set-theoretically, the components completely determine the quotient map $q$.

Conversely, given quotient data $(q,K,\Delta)$ for $E\to M$, there is an induced foliation $\F=\F_{(q,K,\Delta)}$ on the total space $E$ defined by
$$  \ E_x \ni e_x \text{ is on the same leaf as } e'_y \in E_y \iff q_M(x)=q_M(y) \ \text{and } \Delta_{x,y}(e'_y \ mod \ K_y) = e_x \ mod \ K_x .$$
It is easy to verify that this foliation is regular. We sometimes consider the underlying equivalence relation $e \sim e'$ when $e,e' \in E$ are in the same leaf of $\F$.

\begin{example}[Pullback bundles and $K=0$] \label{ex:pullbackquot}

	Given a vector bundle $\tilde E \to \tilde M$ and a simple quotient $q_M:M \to \tilde M$, there is a natural set of quotient data $(q_M,K,\Delta)$ for the pull-back bundle $E:=q_M^*\tilde E \to M$, defined as follows. Recall that, as a set,
	$$ q_M^*\tilde E = \{ (x, \tilde e_{\tilde x}): x \in M, \ \tilde e_{\tilde x} \in \tilde E|_{\tilde x}, \ q_M(x)=\tilde x \}.$$ 
	The quotient data is 
	$$ K = 0_M \ \text{the zero section,} \ \Delta_{x,y}(y,\tilde e_{q_M(y)}) = (x, \tilde e_{q_M(y)}) \in \tilde E_{q_M(y)}$$
	so that $\Delta$ gives the natural identification between vector fibers in $q_M^*\tilde E$. The quotient of $q_M^*\tilde E$ by the induced foliation $\F=\F_{(q_M,K,\Delta)}$ yields the original $\tilde E$,
	$$ \tilde E \simeq (q_M^*\tilde E) /\F,$$
	where the identification map is induced by the projection map $(x,\tilde e_{q_M(x)}) \mapsto \tilde e_{q_M(x)}$. Conversely, given $(q_M,K,\Delta)$ with $K=0_M$, let $\tE = E/\sim$ where $e_x \sim e'_y \iff e_x=\Delta_{x,y}e'_y$ for $q_M(x)=q_M(y)$. Then, there is a unique vector bundle structure on $\tE$ over $\tM$, with projection $[e_x]\mapsto q_M(x)$, such that $$ E\to q_M^*\tE, e_x \mapsto (x,[e_x]) \ \text{is an isomorphism.}$$
	Indeed, $\tM$ can be covered with domains of local sections $\sigma:\tilde U \subset \tM\to M$ of $q_M$ with $$  \sigma^*E \simeq\tE|_{\tilde U}, \ (\tilde x, e_{\sigma(\tilde x)}) \mapsto [e_{\sigma (\tilde x)}].$$
	This last fact follows since, having the reference point $\sigma(q_M(x))$ in the $q_M$-fiber of $q_M(x)$ implies $e'_y \sim e_{\sigma (q_M(x))} \ \iff \ e'_y = \Delta_{y,\sigma(q_M(x))} e_{\sigma(q_M(x))}$. See also \cite[Thm. 2.1.2]{mackenziebook}.
\end{example}

Given quotient data $(q_M:M\to \tilde M,K,\Delta)$ for $E\to M$, applying the construction of the previous example to $E/K\to M$, it follows that the foliation $\F=\F_{(q_M,K,\Delta)}$ is simple and that there exists a unique vector bundle structure on $E/\F$ over $\tilde M$ such that the quotient map $q: E \to E/\F$ defines a simple quotient of $E\to M$ with $\F_{(q_M,K,\Delta)}=\F(q)$. We summarize the discussion in the following:
\begin{proposition} 
	The assignment taking a simple quotient to its components, $q \mapsto (q_M,K,\Delta)$ as defined in eq. \eqref{eq:qcomponents}, defines a 1-1 correspondence between simple quotients of $E\to M$ and quotient data for $E\to M$.
\end{proposition}

\begin{remark}[Splitting and coordinate description] \label{rmk:kerTqcoord}
	Notice that when $q$ is a simple quotient of $E\to M$ with underlying components $(q_M,K,\Delta)$, then there exists a bundle isomorphism
	$$ E \simeq K \oplus_M q_M^*\tilde E,$$
	parametrized by linear splittings of the epimorphism $E \to E/K\simeq q_M^*\tE$. For later reference, we also give here an explicit description of tangent to the $q$-fibers, $Ker(Tq)\subset TE$, using adapted coordinates. Given a basis of local sections $\{e_a:U\subset M\to E\}, $ there is a naturally induced trivialization
	$$\psi: U \times \R^{rk(E)} \to E|_U, \ (x,u) \mapsto u^a e_a(x). $$
	We choose the basis in the form $\{e_k\}\cup \{e_\a\}$ where the first elements $e_k$ span $K$ and that the rest $e_\a$ are $q_M$-invariant, meaning $e_\a(x)=e_\a(y)$ if $q_M(x)=q_M(y)$, and induce a basis for $E/K$. In these coordinates,
	$$ Ker(Tq|_{E|_U}) = T\psi\left(  \{(\dot x, \dot u)\in T_xU \times \R^{rk(E)}: D_{x}q_M(\dot x)=0, \dot u^\a=0 \} \right).$$
\end{remark}

\begin{remark}[Basic vector bundles vs. basic differential forms]
	For a vector bundle to be ``$q_M$-basic" (as in Example \ref{ex:pullbackquot}) the infinitesimal condition of having a flat connection along the $q_M$-fibers is not enough, one needs the representation $\Delta$ of the submersion groupoid $M\times_{\F(q_M)} M$ which allows to identify different fibers. For functions or differential forms, the criterion is simpler: since the $q_M$-fibers are connected, it is enough to verify it infinitesimally through contraction and Lie derivative as recalled in Section \ref{subsec:general}.
\end{remark}

We now provide some examples to be used later.
\begin{example}\label{ex:normalquot}
	Given a surjective submersion $q_M: M\to \tM$ and taking $E=TM$, we describe here quotient data for $E$ such that the corresponding quotient yields $\tE= T\tM$. To this end, we consider $K=Ker(Tq_M) \subset TM$ and we notice that $E/K \simeq q_M^*(T\tM)$ where the identification is given by $$ e_x \ mod \ Ker(T_xq_M) \mapsto (x,T_xq_M(e_x)), \ e_x \in T_xM.$$
	Using this identification, we induce a $\Delta^\nu$ from the one given on $q_M^*(T\tM)$ in Example \ref{ex:pullbackquot}. We denote the corresponding simple quotient by $$q_M^\nu:TM \to T\tM.$$
\end{example}

\begin{example}\label{ex:conormalquot}
	The dual picture of the previous example goes as follows. Given a surjective submersion $q_M: M\to \tM$, we consider $E=\mathrm{Ann}(Ker(Tq_M))\subseteq T^*M$ the annihilator of the $q_M$-fiber directions. We take $K=0_M \subset E$ the zero section and notice that $E \simeq q_M^*(T^*\tM)$ where the identification is given by $$ \tilde \alpha \mapsto q_M^*\tilde \alpha, \ \tilde \alpha \in T^*\tM.$$
	Using this identification, we induce a $\Delta^{\nu^*}$ from the one given on $q_M^*(T^*\tM)$ in Example \ref{ex:pullbackquot}. We denote the corresponding simple quotient by
	$$q_M^{\nu^*}:\mathrm{Ann}(\ker(Tf)) \to T^*\tM.$$
\end{example}
Now we proceed to discuss how the quotient of vector bundles described before can be extended to morphisms. 
\begin{definition}\label{def:basicvbmorph}
	Let $(q_j,q_{M_j})$ be a simple quotient of $E_j\to M_j$, for $j=1,2$. We say that a vector bundle morphism 
	$$ (\Phi,\phi): (E_1\to M_1) \to (E_2\to M_2)$$
	is {\bf $(q_1,q_2)$-basic} if there exists a vector bundle morphism 
	$$ (\tilde \Phi,\tilde \phi): (q_1)_*E_1 \to (q_2)_*E_2$$
	such that $q_2 \circ \Phi = \tilde \Phi \circ q_1$.
\end{definition}
In this case, the morphism $\tilde \Phi$ is unique and we denote it by
$$ \tilde \Phi = (q_1,q_2)_*\Phi.$$
We now characterize basic morphisms in terms of the underlying components $(q_{M_j}:M_j\to \tM_j ,K_j,\Delta^j)$ of $q_j$, for $j=1,2$.
\begin{lemma} \label{lma:basicmorphvb}
With the notations above, a morphism $(\Phi,\phi)$ is $(q_1,q_2)$-basic iff
\begin{enumerate}
	\item $\phi(\F(q_{M_1}))\subseteq \F(q_{M_2})$, uniquely defining a map $\tilde \phi:\tilde M_1 \to \tilde M_2$ via
	$$ q_{M_2}\circ \phi = \tilde \phi \circ q_{M_1},$$
	\item $\Phi(K_1) \subseteq K_2$ so that there is an induced vector bundle morphism $\bar \Phi: (E_1/K_1) \to (E_2/K_2)$ covering $\phi$,
	\item $\bar \Phi$ commutes with the $\Delta$'s:
	$$ \bar \Phi \circ \Delta^1_{x,y} = \Delta^2_{\phi(x),\phi(y)} \circ \bar \Phi , \text{whenever} \ q_{M_1}(x)=q_{M_1}(y).$$
	Moreover, since the $q_M$-fibers are connected, (3.) above is equivalent to its infinitesimal version: $\bar \Phi\circ \bna^1-(\phi^*\bna^2) \circ \bar \Phi =0$, where $\bna^k$ denote the differentiation of $\Delta^k$, for $k=1,2$, as in Remark \ref{rmk:integrablebna}, and we denoted $\bar \Phi: \Gamma_{M_1} (E_1/K_1) \to \Gamma_{M_1}(\phi^*(E_2/K_2))$ the induced map on sections and $\phi^*\bna^2$ the pullback connection.
\end{enumerate}
\end{lemma}

\begin{remark}[About the rank of $\tilde \Phi$]\label{rmk:rankstphi}
	Using the formula $rk(g\circ f)= rk(f)-dim(Ker(g)\cap Im(f))$ for composition of linear maps, we obtain that the rank of $\tilde \Phi$ at a point $q_1(x)$ is given by
	$$ rk_{q_1(x)}(\tilde \Phi)=rk_x(\Phi)-dim(K_2|_{\phi(x)}).$$
	In particular, if $\Phi$ has constant rank, then $\tilde \Phi$ has constant rank.
\end{remark}
We end this subsection by introducing a notion that will be used in the sequel. A section $e\in \Gamma(E)$ is called \textbf{$q$-invariant} for a simple quotient $q$ with components $(q_M,K,\Delta)$ if $$\Delta_{y,x}(e_x \ mod \ K_x)= e_y \ mod \ K_y\text{ whenever }q_M(x)=q_M(y).$$
Notice that a section $e\in \Gamma(E)$ is $q$-invariant if and only if there exists a section $\tilde e\in \Gamma(\tE)$ of the induced quotient bundle $\tE=q_*E$ such that
$$ e_x \ mod \ K_x \equiv \tilde e_{q_M(x)} , \text{ under} \ (E_x/K_x)\simeq \tE_{q_M(x)} .$$
An invariant section $e$ determines a quotient section $\tilde e$ uniquely, and we write $$\tilde{e}=q_*e.$$
For each $\tilde{e}$ there exists at least one invariant section $e$ descending to $\tilde e$.

\begin{remark}\label{rmk:localinvarbasis} 
	For every point in $M$, there exists a neighborhood supporting a local basis of sections of $E$ formed by $q$-invariant elements. This follows directly from the local description in Remark \ref{rmk:kerTqcoord} and by observing that sections of $K$ are always $q$-invariant, indeed projectable to the zero section of $\tE$. 
\end{remark}


\subsubsection*{Simple quotients of Lie algebroid structures and morphisms}

Recall that a {\bf Lie algebroid structure} on a vector bundle $E\to M$ consists of an anchor map $\rho:E \to TM$ (over $id_M$) and a Lie bracket  $[\cdot,\cdot]:\Gamma(E)\times \Gamma(E)\to \Gamma(E)$ satisfying the Leibniz identity

$$[e_1,fe_2]=f[e_1,e_2]+(L_{\rho(e_1)}f)e_2,$$
for every $e_1,e_2\in\Gamma(E)$ and $f\in C^{\infty}(M)$.

A Lie algebroid structure will be denoted by $A$ and we should think of this as a triple $A=(E,\rho,[\cdot,\cdot])$ as above. Hence $A$ is regarded as an additional geometric structure on the vector bundle $E\to M$. 

In this paper we are concerned with vector bundles $E$ equipped with suitably compatible geometric structures which are basic with respect to simple quotients of $E$. From this perspective, we proceed now to introduce the notion of a basic Lie algebroid structure on a vector bundle.

Consider a Lie algebroid structure  $A$  on the vector bundle $E\to M$ and a simple quotient $q$ of $E\to M$.
\begin{definition}
We say that {\bf $A$ is $q$-basic} if there exists a Lie algebroid structure $\tilde A$ on $\tilde E=q_*E$ such that the quotient map $q$ defines a Lie algebroid morphism. In this case, we also say that $q:A\to \tA$ {\bf defines a simple quotient of $A$}.
\end{definition}
In such a case, $\tilde A$ is unique and we shall refer to it as to
$$ \tilde A = q_* A.$$
The following result characterizes when $A$ is basic in terms of the components of $q$.
\begin{proposition}(\cite{higgins-mackenzie})\label{prop:Abasic}
	Let $q$ be a simple quotient of $E\to M$ with components $(q_M:M\to \tM,K,\Delta)$ and $A$ be a Lie algebroid structure on $E\to M$. Then, $A$ is $q$-basic  iff
		\begin{enumerate}
		\item the anchor $\rho:E \to TM$ is $(q,q_M^\nu)$-basic, where $q_M^\nu:TM\to T\tM$ was defined in Example \ref{ex:normalquot};
		\item if $e \in \Gamma(E)$ is $q$-invariant and $k\in \Gamma(K)$, then $[k,e]\in \Gamma(K)$;
		\item if $e_1,e_2 \in \Gamma(E)$ are $q$-invariant, then $[e_1,e_2]$ is also $q$-invariant.
	\end{enumerate}
	In this case, the quotient algebroid structure $q_*A$ is characterized by
 	\begin{enumerate}
 		\item the anchor map is $\tilde \rho:\tE \to T\tilde{M}$ defined by
 		\begin{equation}\label{eq:reducedanchor}
 		\tilde \rho(e_x \ mod \ K_x)= \rho(e_x) \ mod \ Ker(T_xq_M).
 		\end{equation}
 		
 		\item the Lie bracket $[,]_{\tA}$ is defined as follows: for each pair of sections $\tilde{e}_j\in \Gamma(\tE), \ j=1,2$, 
  		 \begin{equation}\label{eq:reducedbracket}
 		 [\tilde{e}_1,\tilde{e}_2]_{\tA} = q_*[e_1,e_2]
 		 \end{equation} 
 		 for any pair of $q$-invariant sections $e_1,e_2 \in \Gamma(E)$ such that $q_*e_j=\tilde e_j, \ j=1,2$.
 	\end{enumerate}
\end{proposition}

The pair $(K,\Delta)$ in the quotient data making $A$ basic is called an \textbf{ideal system} in \cite{higgins-mackenzie}; the previous Proposition is just a re-phrasing of their result in terms of our terminology.

\begin{remark} \label{rmks:pullbackq} We give some remarks below:
	\begin{itemize}
	\item Let $A$ be $q$-basic and denote $(q_M,K,\Delta)$ the components of $q$. Then, $q$ defines a \textbf{fibration} of Lie algebroids as in \cite[Def. 4.4.1]{mackenziebook} for which the $q_M$-fibers are connected. Also, $Ker(Tq)\subset TA$ defines a morphic subbundle in the sense of \cite{JO}. In particular, $K\subseteq A$ defines a Lie subalgebroid over $M$. Notice, however, that the vector bundle $E/K\to M$ does not inherit a Lie algebroid structure in general. 

	\item Recall (e.g. from \cite{higgins-mackenzie}) that given a surjective submersion $q_M:M \to \tM$ and a Lie algebroid structure $\tA$ on $\tE\to \tM$, there is a \textbf{pullback Lie algebroid} structure $q_M^!\tA$ with underlying vector bundle
	$$ (q_M,\tilde \rho)^!\tE :=\{(v,\tilde e)\in TM\times \tE: Tq_M(v)=\tilde \rho(\tilde e) \} \subset TM \times \tE$$
	and the Lie algebroid structure $q_M^!\tA$ is induced from the product $TM \times \tA$. Let $q:A \to \tA$ be a simple quotient of a Lie algebroid $A$. Then, the map
	$$ \psi: A \to q_M^!\tA, e \mapsto (\rho(e),q(e))$$
	defines a Lie algebroid morphism. This map $\psi$ is an isomorphism iff $\rho|_K: K \to Ker(Tq_M)$ is an isomorphism of vector bundles over $id_M$ (c.f. \cite[Lemma 3.5]{Alv}). In this case, we will say that $q$ is of \textbf{pullback type}. Notice that, given a simple quotient $q_M:M \to \tM$ of $M$ and $\tA$ a Lie algebroid over $\tM$, the natural map
	$$ q: q_M^!\tA \to \tA$$
	defines a simple quotient of pullback type, and $\rho|_K^{-1}(v)=(v,0)$ defines a natural injective morphism $Ker(Tq_M)\to q_M^!\tA$.

\end{itemize}
\end{remark}

Finally, we consider quotients of Lie algebroid morphisms which will play a key role in our description of morphic forms on Lie algebroids. The following result can be deduced from \cite[Thm. 4.4.3]{mackenziebook}.

\begin{proposition}\label{prop:quotAmorph}
	Let $(q_j,q_{M_j})$ be a simple quotient of $E_j \to M_j$ and $A_j$ be a $q_j$-basic Lie algebroid structure on $E_j\to M_j$, for $j=1,2$. Let $(\Phi,\phi):E_1 \to E_2$ be a vector bundle morphism. If $\Phi$ defines a Lie algebroid morphism $A_1 \to A_2$ and $\Phi$ is $(q_1,q_2)$-basic, then $(q_1,q_2)_*\Phi$ defines a Lie algebroid morphism $(q_1)_*A_1 \to (q_2)_*A_2$ between the quotient algebroids.
\end{proposition}

\subsection{Characterizing basic morphic forms on Lie algebroids}

Let $E\to M$ be a vector bundle. A $k$-form $\omega\in \Omega^k(E)$ is called {\bf linear} if

$$ h_\lambda^*\om = \lambda \om, \ \forall \lambda \in \R,$$
where $h_\lambda$ denotes fiberwise scalar multiplication by $\lambda$ on the fibers of $E$. Notice that the exterior differential $d\om$ of a linear form $\om$ is again a linear form.

\begin{example}
	Consider the vector bundle 
	$$p:E = \Lambda^k T^*M \to M$$
	whose sections are $k$-forms on $M$. This bundle comes with a tautological $k$-form $\theta_k\in \Omega^k(E)$ defined by
	$$ \theta_k|_{\a}(V_1,\dots,V_k) = \a(Tp(V_1),\dots,Tp(V_k)), \ \a\in \Lambda^k T^*M, \ V_j \in T_\a (\Lambda^k T^*M), \ j=1,...,k.$$
	It is easy to see that $\theta_k$ is a linear $k$-form. As a consequence of the previous observation, the canonical symplectic form $\om_{can}=d\theta_{1} \in \Omega^2(T^*M)$ is linear.
\end{example}

One easily observes that linear forms can be pulled back along vector bundle morphisms. In particular, a vector bundle morphism $\mu:E \to \Lambda^k T^*M$ over the identity on $M$ induces a linear $k$-form on $E$, defined by
$$\Lambda_\mu : = \mu^*\theta_k \in \Omega^k(E).$$
We now recall the following characterization of linear forms on vector bundles. 
\begin{proposition}(\cite{bursztyncabrera})
	The assignment 
	$$ (\mu:E\to \Lambda^{k-1} T^*M, \eta: E\to \Lambda^{k} T^*M) \mapsto d\Lambda_\mu + \Lambda_\eta$$
	defines a 1-1 correspondence into the set of linear $k$-forms on $E$.
\end{proposition}

As a consequence, a linear $k$-form $\omega\in\Omega^k(E)$ can be written as $\omega= d\Lambda_\mu + \Lambda_\eta$ for unique vector bundle morphisms $\mu:E\to \Lambda^{k-1} T^*M, \eta: E\to \Lambda^{k} T^*M$. We refer to $(\mu,\eta)$ as the \textbf{components} of $\omega$. The components can be recovered from $\om$ via the following formulas: for $e\in E_x$, $v_1,\dots,v_k\in T_xM$,
\begin{eqnarray}
\mu(e)(v_1,\dots,v_{k-1})&=& \om|_{0_x}(\frac{d}{dt}|_{t=0}(te), T0(v_1),\dots,T0(v_{k-1})) \\
\eta(e)(v_1,\dots,v_{k})&=& (d\om)|_{0_x}(\frac{d}{dt}|_{t=0}(te), T0(v_1),\dots,T0(v_{k}))
\end{eqnarray}
where $0:M \to E$ denotes the zero section.
\begin{remark}\label{rmk:omlinpulle}
	It is straightforward to verify that, for a linear $k$-form $\om$ with components $(\mu,\eta)$ and any section $e\in \Gamma_M(E)$, then the following holds:
	$$ e^*\om = d \mu(e) + \eta(e) \in \Omega^k(M).$$
\end{remark}
We now characterize when a linear form is basic for a simple quotient in terms of the underlying components.
\begin{proposition}\label{prop:linomqbasic}
	Let $q$ be a simple quotient of $E\to M$ with components $(q_M,K,\Delta)$ and let $\om$ be a linear $k$-form on $E\to M$ with components $(\mu,\eta)$. Then, $\om$ is $q$-basic iff
	\begin{enumerate}
		\item $K\subset Ker(\mu)$ and $K\subset Ker(\eta)$,
		\item for each $q$-invariant section $e\in \Gamma(E)$, the forms
		$$ \mu(e)\in \Omega^{k-1}(M), \ \eta(e)\in \Omega^{k}(M) \text{ are $q_M$-basic.}$$
	\end{enumerate}
	In this case, $q_*\om$ is a linear $k$-form on the quotient bundle $q_*E$ and the components $(\tilde \mu,\tilde \eta)$ of $q_*\om$ are uniquely determined by 
	\begin{eqnarray}
	 \tilde \mu: q_*E \to \Lambda^{k-1}T^*\tM&,& \ [q_M^*\tilde \mu(q(e))]|_x =  \mu(e) \nonumber \\
	 \tilde \eta: q_*E \to \Lambda^{k-1}T^*\tM&,& \ [q_M^*\tilde \eta(q(e))]|_x = \eta(e) \label{eq:tildemuetadef}
	 \end{eqnarray}
	 for every $e\in E|_x, \ x\in M$.
\end{proposition}

\begin{proof}
	The statement that $\om$ is $q$-basic can be verified locally around each point of $E$. Around each point of $E$, we can use a coordinate system $(x^i,u^a)$ of the form given in Remark \ref{rmk:kerTqcoord}, with underlying basis of local sections $\{e_k\}\cup \{e_\a\}$ where $e_k$ span the fibers of $K\subset E$ and $e_\a$ are $q$-invariant sections. In this sytem, we have
	$$ \om = du^a \wedge \mu(e_a) + u^a (d\mu(e_a)+\eta(e_a)), \ d\om=  du^a \wedge \eta(e_a) + u^a d\eta(e_a). $$
	Following Remark \ref{rmk:kerTqcoord} further, a vector $V=(\dot x, \dot u)$ is in $Ker(Tq)$ iff $Tq_M(\dot x)=0$ and $\dot u^\alpha =0$, i.e. $\dot u\in K$.
	We then compute
	$$ i_V \om = \dot u^k \mu(e_k) - du^a \wedge i_{\dot x} \mu(e_a) + u^a i_{\dot x} (d\mu(e_a)+\eta(e_a)),$$
	$$ i_V d\om = \dot u^k \eta(e_k) - du^a \wedge i_{\dot x} \eta(e_a) + u^a i_{\dot x} d\eta(e_a).$$
	Then, the conditions $i_V\om = 0$ and $i_V d\om$ are equivalent to 
	$$ \mu(e_k)=0, \ \eta(e_k)=0, \ i_{\dot x}\mu(e_\a) = 0, \ i_{\dot x} \eta(e_\a) =0 , i_{\dot x}d\mu(e_\a) = 0, \  i_{\dot x}d\eta(e_\a) = 0. $$
	Since any invariant section is locally of the form $e = f^k e_k + (q_M^*\tilde g^\a) e_\a$ with $f^k$ functions on $M$ and $\tilde g^\a$ functions defined on $\tM$, the above system of equations is equivalent to statements \emph{1.} and \emph{2.}, finishing the proof.
\end{proof}

If $\om$ has components $(\mu,\eta)$ satisfying items \emph{1.,2.} as in Proposition \ref{prop:linomqbasic}, we say that {\bf $(\mu,\eta)$ are $q$-basic} and we denote 
$$ (\tilde \mu,\tilde \eta)=q_*(\mu,\eta)$$
the {\bf quotient components} corresponding to $\tom =q_*\om$.

We now consider a vector bundle $p:E\to M$ endowed with an algebroid structure $A$. Recall that there is an induced Lie algebroid structure $TA$ on $Tp:TE\to TM$ called the \textbf{tangent algebroid} structure. The tangent Lie algebroid structure can be extended to any direct sum of $TE\to TM$ with itself. The following definition is taken from \cite{bursztyncabrera} where more details can be found.

\begin{definition}
 A linear $k$-form on $E\to M$ is called {\bf $A$-morphic} when the map
 $$ \underline{\om}: T^{\oplus k}E:= TE \oplus_E \dots \oplus_E TE \to \R, \ V_1,\dots, V_k \mapsto \om(V_1,\dots,V_k)$$
 defines a Lie algebroid morphism with respect to the direct sum Lie algebroid structure $T^{\oplus k}A:=TA\oplus_A \dots \oplus_A TA$ on the domain and the Lie algebra structure $\R \to \star$ on the codomain.
\end{definition}

Let $\omega\in \Omega^k(E)$ be a linear $k$-form with components $(\mu,\eta)$. As shown in \cite[Thm. 1]{bursztyncabrera}, $\omega$ is $A$-morphic if, and only if, the following conditions hold:
\begin{eqnarray}
i_{\rho(e_1)}\mu(e_2)& = &- i_{\rho(e_2)}\mu(e_1) \nonumber\\
\mu([e_1,e_2])& =& L_{\rho(e_1)}\mu(e_2)-i_{\rho(e_2)}d\mu(e_1) -
i_{\rho(e_2)}\eta(e_1) \nonumber\\
\eta([e_1,e_2])&=& L_{\rho(e_1)}\eta(e_2) - i_{\rho(e_2)}d\eta(e_1),
\label{eq:IMeqs}
\end{eqnarray}
for all $e_1,e_2 \in \Gamma(E)$. A pair of vector bundle morphisms $(\mu,\eta)$ satisfying the equations \eqref{eq:IMeqs} is called an {\bf IM $k$-form on $A$}. 


\begin{remark}[Restricting morphic forms]\label{rmk:restrictionIM}
	Let $E'\subseteq E$ be a subbundle over $M$ together with a Lie subalgebroid structure $A'\subseteq A$. An $A$-morphic $k$-form $\om$ on $E$ can be restricted to $E'$ yielding an  $A'$-morphic $k$-form $\omega'$. If $(\mu,\eta)$ are the components of $\omega$, then the components $(\mu',\eta')$ of $\omega'$ are given by the restrictions $\mu|_{E'}$ and $\eta|_{E'}$, respectively.
\end{remark}

\medskip

The following result states that being morphic is preserved under simple quotients.

\begin{theorem}\label{thm:Aandomqbasic}
	Let $E \to M$ be a vector bundle and $q$ a simple quotient of $E\to M$. Let $A$ be a Lie algebroid structure on $E\to M$ and $\om$ a linear $k$-form which is $A$-morphic. If $A$ and $\om$ are $q$-basic, then $q_*\om$ is $(q_*A)$-morphic. In particular, the quotient components $(\tmu,\teta)$ underlying $q_*\om$, which are given by \eqref{eq:tildemuetadef} in terms of the components of $\om$, satisfy the IM equations \eqref{eq:IMeqs} defining an IM $k$-form on $q_*A$.
\end{theorem}

\begin{proof}
	The proof consists of applying Proposition \ref{prop:quotAmorph} to the algebroid morphism $\underline{\om}$. To this end, we will show the following steps:
	\begin{enumerate}
		\item $q$ induces a simple quotient $$\hat q: T^{\oplus k}E \to T^{\oplus k}\tilde E, (V_1,\dots, V_k) \mapsto (Tq(V_1),\dots, Tq(V_k)),$$ where $\tilde E=q_*E$;
		\item the Lie algebroid structure $T^{\oplus k}A$ on  $T^{\oplus k}E$ is $\hat q$-basic and $$ \hat q_*(T^{\oplus k}A)=T^{\oplus k} \tA,$$
		where $\tA=q_*A$ is the quotient Lie algebroid structure;
		\item since $\om$ is $q$-basic, then $\uom$ seen as a vector bundle morphism is $(\hat q,id_\R)$-basic;
		\item $(\hat q,id_\R)_*\uom = \underline{(q_*\om)}$.
	\end{enumerate}
When the above steps are proven, the present Proposition follows directly from Proposition \ref{prop:quotAmorph} applied to the algebroid morphism $\underline{\om}:T^{\oplus k}A \to \R$.

Item (1.) follows directly from the split normal form of $E$ relative to $q$ in Remark \ref{rmk:kerTqcoord}. 

For item (2.), we will follow the computations in \cite{bursztyncabrera} involving tangent lifted algebroid structures. We first argue that if $A$ is $q$-basic, then $TA$ is $Tq$-basic. Since this statement is of infinitesimal nature, we can restrict to a local chart $(x^i,u^a)$ of $E$ as in Remark \ref{rmk:kerTqcoord}, where the underlying local basis of sections of $E$ is split as $\{e_a\}=\{e_k\}\cup \{e_\a\}$, with $e_k$ spanning $K\subset E$ and $e_\a$ being $q$-invariant. Given any (local) section $e$ of $E$, there are induced sections, $\hat e, Te:TM\to TE$, locally defined by 
$$ \hat e_a (x,\dot x) = (x,\dot x, u=0, \dot{u}^b= \delta^b_a), \ Te_a(x,\dot x) = (x,\dot x, u^b=\delta^b_a, \dot u =0),$$
where $(x,\dot x, u, \dot u)$ are the natural induced local coordinates on $TE$.

The Lie algebroid structure $TA$ on $TE\to TM$, with anchor denoted $\rho_{TA}$ and bracket $[\cdot,\cdot]_{TA}$ is uniquely characterized by the local expressions
\begin{eqnarray} \rho_{TA}(\hat e_a) = \rho^i_a \partial_{\dot x^i}, \ \rho_{TA}(Te_a)= \rho^j_a \partial_{x^j} + \dot x^i (\partial_{x^i}\rho_a^j) \partial_{\dot x^j}, \nonumber \\
  \ [ \hat{e}_a,\hat e_b ]_{TA} =0, \ [Te_a,\hat e_b]_{TA} = C^c_{ab} \hat e_c , \ [Te_a,T e_b]_{TA} = C^c_{ab} T e_c + \dot{x}^i (\partial_{x^i}C^c_{ab}) \hat e_c \nonumber
 \end{eqnarray}
where $\rho(e_a)=\rho_a^i \partial_{x^i}$ and $[e_a,e_b]= C^c_{ab} e_c$ define the local structure functions of $A$. See details in \cite{bursztyncabrera} and references therein.

Splitting the $x^i$ coordinates into $\{y^l\}\cup \{\tilde x^n\}$ so that $q_M(x)=\tilde x$, then $q$ is uniquely defined by
$$ q(e_k(x))=0, \ q(e_\a(x))= \bar e_\a(\tilde x),$$
where $\bar e_\a$ is the induced local basis of sections for $\tE \to \tM$. With these notations, and the characterization in Proposition \ref{prop:Abasic} of the $q$-basic algebroid structure $A$ (re-expressed in the present coordinates), it follows directly that $TA$ is $Tq$-basic. Next, it is also immediate from the coordinate characterization of the algebroid structure $T^{\oplus k}A$ given in \cite[\S 3.2]{bursztyncabrera}, that item (2.) holds.

For (3.), we need to show that the items in Lemma \ref{lma:basicmorphvb} hold for $\underline{\om}:T^{\oplus k}E \to \R$ (whose base map is $T^{\oplus k}M \to \star$) with respect to the pair of simple quotients $(\hat q, id_\R)$. These boil down to,
$$ \text{(3a.) } \underline{\om}(V_1,\dots, V_k) = 0 \text{ when $Tq(V_j)=0$ for all $j$, }$$
and 
$$ \text{(3b.) } \underline{\om}(V_1,\dots, V_k)=  \underline{\om}(V'_1,\dots, V'_k)\text{ when $Tq(V_j)=Tq(V_j')$ for all $j$. }$$
Now, both $(3a., 3b.)$ follow direclty from $\om = q^*\tilde \om$ being $q$-basic by hypothesis.

Finally, for item (4.), we use the characterization of the quotient morphism $\tilde \Phi:=(\hat q,id_\R)_*\uom$ (see Definition \ref{def:basicvbmorph}) yielding $\tilde \Phi \circ \hat q = \uom$, so that
$$ \tilde \Phi(Tq(V_1),\dots,Tq(V_k)) = \uom(V_1,\dots, V_k).$$
Using $\om=q^*\tom$ by hypothesis, we see that $(\hat q,id_\R)_*\uom = \underline{\tom}$ as wanted. This finishes the proof of the theorem.
\end{proof}

\begin{example}[Basic $k$-forms as basic IM $k$-forms]
	Let $q_M:M\to \tilde M$ be a surjective submersion with connected fibers and consider $q_M^\nu:E:=TM \to T\tM$ the induced normal simple quotient described in Example \ref{ex:normalquot}. A closed $k$-form $\omega_M\in\Omega^k(M)$ can be seen as an IM $k$-form $\mu:=\omega_M^{\flat}:TM\to \Lambda^{k-1}T^*M$ on the tangent algebroid $A=TM$. Following Proposition \ref{prop:linomqbasic}, it is immediate to verify that the IM $k$-form $\omega_M^\flat$ is $q_M^\nu$-basic iff $\om_M$ is $q_M$-basic in the standard sense. In this case, the quotient IM $k$-form $q_*\mu = \tilde \mu$ is given by $\tilde \om_M^\flat$, where $\om_M=q_M^* \tilde \om_M$.
\end{example}

\begin{example}[Quotient of $G$-equivariant IM $2$-forms]\label{ex:Gequivar} Let $G$ be a Lie group and $E\to M$ be a $G$-equivariant vector bundle equipped with a $G$-invariant Lie algebroid structure $A$. Suppose that $\mu:E\to T^*M$ is a $G$-equivariant IM $2$-form, where $T^*M$ is equipped with the cotangent lift of $G\curvearrowright M$. It was shown in \cite{bursztyncabrera2} that $\mu$ can be reduced to an IM $2$-form on a quotient Lie algebroid of the form $B/G$, for a certain Lie subalgebroid $B$ of $A$. Here we explain how this reduction procedure involves, as one of its steps, the quotient procedure of Theorem \ref{thm:Aandomqbasic}. 
	For that, assume that $G$ acts freely and properly on $M$, so we have a surjective submersion $q_M:M\to M/G=\tM$ defining a principal $G$-bundle. 
	In this case, one has that $\ker(Tq_M)$ is given by the infinitesimal action of $\gg=Lie(G)$ at each point of $M$ and the cotangent lifted action $G\curvearrowright T^*M$ is hamiltonian with moment map $j:T^*M\to \mathfrak{g}^*$ given by the dual of the infinitesimal action map. In \cite{bursztyncabrera2}, it is proven that the composition $J:=j\circ \mu:E\to \mathfrak{g}^*$ is a Lie algebroid morphism and that, whenever smooth, $(E':=J^{-1}(0)\to M) \subset (E\to M)$ defines a Lie subalgebroid $B$ of $A$. Since $\mu$ is $G$-equivariant, then $B$ carries a canonical $G$-action which induces a simple quotient $q:E'\to \tE=J^{-1}(0)/G$ with components $(q_M,K=0,\Delta)$ where $\Delta$ is given by the $G$-action on $q_M$-fibers. Moreover, $B$ is $q$-basic. In this setting, the restriction $\mu_B:=\mu|_B:B\to T^*M$ automatically defines an IM $2$-form $(\mu_B,\eta=0)$ on $B$ (see Remark \ref{rmk:restrictionIM}) with the following properties:
	\begin{enumerate}
		\item $\mathrm{im}(\mu_B)\subseteq \mathrm{Ann}(\ker(Tq_M))=J^{-1}(0),$ 
		\item $\mu_B$ seen as a vector bundle map $B\to \mathrm{Ann}(\ker(Tq_M))$ is equivariant with respect to the given $\Delta$ on $B$ and to the conormal $\Delta^{\nu^*}$ on $\mathrm{Ann}(\ker(Tq_M))$ given in Example \ref{ex:conormalquot}.
	\end{enumerate}
	Following Proposition \ref{prop:linomqbasic}, it is immediate to show that $(\mu_B,0)$ is $q$-basic. By Theorem \ref{thm:omActrk}, the quotient $q_*(\mu_B,0)$ defines an IM $2$-form for the quotient algebroid structure $\tilde B=q_*B$ on $(\tE = J^{-1}(0)/G \to M/G)$. It is straighforward to check that this quotient IM $2$-form coincides with the \emph{reduction of $\mu$ by $G$} as described in \cite{bursztyncabrera2}.
\end{example}

\subsection{Constant rank morphic forms and quotient by their kernels}\label{subsec:ctrkA}

In this subsection, we consider morphic forms of constant rank satisfying extra regularity conditions and proceed to obtain simple quotients by their kernels. We only consider the case of linear $k$-forms with $k\geq 2$, as the other cases $k=0,1$ only yield trivial quotients due to linearity.

The main object of study in this subsection is given in the following definition 


\begin{definition}
	A linear $k$-form $\om$ on $E\to M$ with $k\geq 2$ is called {\bf kernel-reducible} when
	\begin{enumerate}
		\item $Ker(\om)\subseteq TE$ has constant rank as a distribution on $E$,
		\item $Ker(\om)\subset Ker(d\om)$,
		\item the quotient map $q:= q_{Ker(\om)}$ into the leaf space of the foliation integrating $Ker(\om)\subseteq TE$ defines a simple quotient of $E\to M$.
	\end{enumerate}
\end{definition}
It is clear from the definitiontt that a kernel-reducible $k$-form $\omega$ is always $q_{Ker(\om)}$-basic, hence defining a linear form $q_{Ker(\om) *}\om$ on $q_{Ker(\om) *}E$.
Denoting $(\mu,\eta)$ the components of $\om$, we also say that $(\mu,\eta)$ is kernel-reducible when $\om$ is.

We begin by characterizing when a linear form has constant rank and when the condition $Ker(\om)\subset Ker(d\om)$ holds (see Section \ref{subsec:general}), in terms of the components of $\om$.
\begin{lemma}\label{lem:keromlin}
Let $\omega$ a linear $k$-form on $E\to M$ with components $(\mu,\eta)$, $k\geq 2$. Then, $\om$ has constant rank iff the following conditions hold:
	\begin{enumerate}
	\item $\mu: E \to \Lambda^{k-1}T^*M$, as a vector bundle morphism over $id_M$, has constant rank;
	\item the vector bundle morphism over $id_M$ given by
	$$ \mu^\sharp: TM \to E^*\otimes \Lambda^{k-2}T^*M, \ \mu^\sharp(v)(e,v_1,\dots,v_{k-2})=\mu(e)(v,v_1,\dots, v_{k-2}), \ e\in E_x, v,v_j\in T_xM,$$
	has constant rank;
	\item $i_v(d\mu(e)+\eta(e)) \in \Gamma( Im(\mu) )$ for every $v\in \Gamma(Ker(\mu^\sharp))$ and $e\in \Gamma E$.
\end{enumerate}
In this case, $Ker(\om)\subset Ker(d\om)$ iff the following hold:
	\begin{enumerate}
	\item[4.] $Ker(\mu)\subset Ker(\eta)$ and $Ker(\mu^\sharp)\subset Ker(\eta^\sharp)$;
	\item[5.] $Ker(\mu^\sharp)\subset TM$ is involutive as a distribution on $M$;
	\item[6.] the partial connection defined by
	$$ \bna: \Gamma(Ker(\mu^\sharp)) \times \Gamma(E/Ker(\mu)) \to \Gamma(E/Ker(\mu)), \ \mu\left(\bna_v(e \ mod \ Ker(\mu))\right)=i_vd\mu(e),$$
	is flat;
	\item[7.] if $\bna_v(e \ mod \ Ker(\mu)) =0$ then $i_vd\eta(e)=0$.
\end{enumerate}
\end{lemma}

\begin{proof}
	We begin with the conditions \emph{1.- 3.} for $\om$ having constant rank. At a point $0_x\in E$, the tangent space $T_{0_x}E $ splits naturally as
	$$ E_x \oplus T_x M \simeq T_{0_x}E , (e,v)\mapsto \frac{d}{dt}|_{t=0}(te) + T0(v),$$
where $0:M\to E$ denotes the zero section.
	Using this decomposition, it follows from the definition of $(\mu,\eta)$ that (see also eq. \eqref{eq:Vkercoord} below)
	\begin{equation}\label{eq:Keromzero} T_{0_x}E\supset Ker(\om|_{0_x}) \simeq Ker(\mu|_{E_x}) \oplus Ker(\mu^\sharp|_{T_xM}).\end{equation}
	
	Let us now show that when $Ker(\om)$ has constant rank, then $Ker(\mu^\sharp)$ has constant rank, following an argument in \cite[Lemma 3.5]{J}. We observe that $Ker(\mu^\sharp)$ is given by the intersection of two subbundles, $Ker(\om)$ and the tangent to the zero section $T0(TM)$, inside $TE|_{0(M)}$. Since $\om$ is linear, it is easy to verify (see formulas below) that if $V\in Ker(\om|_e), \ e\in E,$ then its projection $Tp(V)\in T0(TM)$ is in $Ker(\om|_{0(p(e))})$, where $p:E\to M$ is the bundle projection. From this follows that $Ker(\mu^\sharp)$ is smooth as a distribution, i.e. every element in $Ker(\mu^\sharp)$ can be extended to a smooth section of $TE|_{0(M)}$ which takes values in $Ker(\mu^\sharp)$. A general result states that, when the intersection of two subbundles is smooth, then the intersection has constant rank. This shows that $Ker(\mu^\sharp)$ has constant rank, and by \eqref{eq:Keromzero}, then $Ker(\mu)$ has constant rank as well. This proves that $Ker(\om)$ having constant rank implies (1.,2.).
	
	We now consider a point $e_0\in E$ away from $0(M)$ and study the rank of $Ker(\om|_{e_0})$. To that end, take a coordinate chart $(U\subset M, x^j)$ around $x_0:=p(e)$ and a local basis of sections $\{e_a:U \to E\}$ defined around $x_0$. With respect to these coordinates, $E|_U\simeq U \times \R^{rk(E)}$, as in Remark \ref{rmk:kerTqcoord}, inducing coordinates $(x^j,u^a)$ on $E|_U$. In these coordinates, the linear form $\om$ reads
	\begin{equation*}
	\om|_{E|_U} = du^a \wedge \mu(e_a) + u^a d\mu(e_a) + u^a \eta(e_a),
	\end{equation*}
	where $\mu(e_a),\eta(e_a)\in \Omega(U)$ are seen as basic forms on $E|_U\to U$. A tangent vector $T_e E\ni V\simeq (v,\dot u)\in T_{p(e)}U\times \R^{rk(E)}$ satisfies $i_V \om = 0$ iff
	\begin{equation}\label{eq:Vkercoord} i_v\mu(e_a) = 0, \forall a; \ \dot u^a \mu(e_a) = - u^ai_v(d\mu(e_a)+\eta(e_a)).\end{equation}
	Notice that the decomposition \eqref{eq:Keromzero} becomes transparent from the above coordinate expression taking $u^a=0$ for every $a$.
	As a consequence of eq. \eqref{eq:Vkercoord}, the rank of $Ker(\om|_e)$ is equal to the rank of $Ker(\om|_{0(p(e))})$ for each $e\in E|_U$ iff
	$$ i_v(d\mu(e_a)+\eta(e_a)) \in Im(\mu|_{E_{p(e_a)}}), \forall a, \forall v\in Ker(\mu^\sharp).$$
	Since $\{e_a\}$ spans a local basis of sections of $E$ and $i_v\mu(e)=0$ for $v\in Ker(\mu^\sharp)$, then the above equation is equivalent to condition (3.) for arbitrary sections $e$. These arguments show that $Ker(\om)$ having constant rank is equivalent to conditions $(1.-3.)$.
	
	For the second part, let us assume $Ker(\om)\subset Ker(d\om)$ and recall from Section \ref{subsec:general} that $Ker(\om)\subset TE$ defines an involutive distribution on $E$. We observe that, in the coordinates $(x^j,u^a)$ used above,
	$$ i_{(v,\dot u)} d\om = 0 \iff i_v\eta(e_a) = 0, \forall a; \ \dot u^a \mu(e_a) = - u^ai_v(d\eta(e_a)).$$
	Then,  $Ker(\mu^\sharp)\subset Ker(\eta^\sharp)$ and, from the case $u^a=0$, $Ker(\mu)\subset Ker(\eta)$, showing item (4.) in the Lemma. With these conditions, eq. \eqref{eq:Vkercoord} reduces to
	$$i_v\mu(e_a) = 0, \forall a; \ \dot u^a \mu(e_a) = - u^a\ i_vd\mu(e_a). $$
	Now, let us discuss the partial connection $\bna$ of point (6.). It is clear from the definition that $\bna$ is a derivation on its second argument. Suppose that the basis $\{e_a\}$ is taken as in Remark \ref{rmk:kerTqcoord} so that the first $rk(Ker(\mu))$-elements define a basis of $Ker(\mu)$ and use the index $\{e_\a \}$ for the other elements, which span a complement to $Ker(\mu)$. Notice that this means that $\{\mu(e_\a) \}\subset \Omega(U)$ is linearly independent and $\{e_\alpha\}$ induces a basis of sections for $E/Ker(\mu)$. In the induced coordinates for $E|_U$, we can thus write the Christoffel symbols for $\bna$ as
	$$ \bna_v e_\a = v^j \Gamma_{\a j}^\b(x) e_\b, \overset{def. \ \bna}{\iff} v^j\Gamma_{\a j}^\b \mu(e_\b) = i_v d\mu(e_\a), \ \forall \a.$$
	We now use them to characterize the kernel of $\om|_{E|_U}$. For each $v \in Ker(\mu^\sharp|_x), \ x\in U$, and each $u$, we define 
	$$(x,u,v)\mapsto \dot u^\a(x,u,v) = - u^\b v^j \Gamma^\a_{\b j}(x),$$
	and take the $\dot u^a$ corresponding to the first $Ker(\mu)$-indices $a$ arbitrarily (eg: equal to zero). 
	Then,
	$$ i_{(v,\dot u(x,u,v))}\om|_{(x,u)} = 0, \ \forall (x,u)\in E|_U, v\in Ker(\mu^\sharp|_x),$$
	because of the updated version of eq. \eqref{eq:Vkercoord} above and the definition of $\bna$.
	
	With these computations, to see that $Ker(\mu^\sharp)\subset TM$ is involutive, we take two sections $v_1,v_2 \in \Gamma(Ker(\mu^\sharp))$ seen as vector fields and we want to check $[v_1,v_2]|_{x_0}$ is in $Ker(\mu^\sharp)$. In our coordinate chart, we can extend $v_k$ to a vector field $V_k(x,u)=(v_k(x), \dot u(x,u,v_k(x)))\in Ker(\om|_{E|_U}), k=1,2,$ as above. Since $Ker(\om)$ is involutive, $i_{[V_1,V_2]}\om =0$. As stated before, the linearity of $\om$ implies that $Tq([V_1,V_2]) \in Ker(\mu^\sharp)$. But, by the construction of $V_k$'s, $Tq([V_1,V_2])|_{x_0}=[v_1,v_2]|_{x_0}$, thus finishing the proof of (5.).
	
	To prove (6.), we take $V_1, V_2$ as above and use the fact that $[V_1,V_2]\in Ker(\om|_{E|_U})$. Writting 
	$$ [V_1,V_2] = ([v_1,v_2], Z), \text{ then } Z^\b\mu(e_\b) = -u^\a i_{[v_1,v_2]}d\mu(e_\a), \ \forall u,$$
	as a consequence of the updated eq. \eqref{eq:Vkercoord}. By direct computation of $Z$ using Lie brackets, the above equation results equivalent to
	$$ -\bna_{v_1}\bna_{v_2}e_\a + \bna_{v_2}\bna_{v_1}e_\a = -\bna_{[v_1,v_2]} e_\a, \forall \a,$$
	showing that $\bna$ must be flat as in (6.).
	
	To prove (7.), the idea is to take the basis elements $\{e_\b \}$ to be $\bna$-flat, which is now possible by (6.). With this choice, $\dot u^\a(x,u,v) =0$ (equiv. $\Gamma^\a_{\b j}=0$).  Since $(v,\dot u(x,u,v))\in Ker(\om|_{E|_U})$, by hipothesis we must have $ i_{(v,\dot u(x,u,v))} d\om = 0$. By our coordinate expression for $Ker(d\om)$ above, this last condition is equivalent to
	$$ u^\a i_v d\eta(e_\a) = 0, \forall u.$$
	Since $i_v \eta(e) =0$ for $v\in Ker(\mu^\sharp)\subset Ker(\eta^\sharp)$, we get that (7.) must hold for all sections $e$.
	
	Finally, the converse implication (1.-7.)$\Rightarrow Ker(\om)\subset Ker(d\om)$ follows directly from the local expressions by considering, at each point, a coordinate chart $U$ adapted to the foliation defined by $Ker(\mu^\sharp)$ and a flat basis $\{e_{\a}\}$ of $E/K|_U$ as above.
\end{proof}

\begin{remark}
	The triple $(q_M,Ker(\mu),\bna)$ defines an infinitesimal ideal system for $E\to M$, seen as a Lie algebroid with the zero anchor and bracket, as defined in \cite{JO}. The corresponding linear involutive distribution is $Ker(\om)$ (see \cite[Thm. 4.5]{JO}).
\end{remark}

We collect the conditions of the previous Lemma into a definition.

\begin{definition}\label{def:infinitkerred}
	Let $\om$ a linear $k$-form on $E\to M$ with components $(\mu,\eta)$, with $k\geq 2$. If $\om$ satisfies all the conditions in Lemma \ref{lem:keromlin}, we say that $\om$ is (equiv. its components $(\mu,\eta)$ are) {\bf infinitesimally kernel-reducible}.
\end{definition}

Observe that, when $\om$ is kernel-reducible thus defining a simple quotient $q_{Ker(\om)}\equiv (q,q_M)$, the foliation $\F_M=\F(q_M)$ integrates $Ker(\mus)\subset TM$ and the components $(q_M, K,\Delta)$ of $q_{Ker(\om)}$ satisfy: $K=Ker(\mu)$ and $\Delta$ integrates $\bna$ as in Remark \ref{rmk:integrablebna}.
\begin{proposition}\label{prop:qomred}
	Let $\om$ be a linear $k$-form on $E\to M$ with $k\geq 2$. Then, $\om$ is kernel-reducible iff $\om$ is infinitesimally kernel-reducible, the foliation $\F_M$ integrating the involutive distribution $Ker(\mu^\sharp)\subset TM$ is simple and the flat partial connection $\bna$ integrates to a representation $\Delta$ of $M\times_{\F_M} M$ on $E/Ker(\mu)$ (see Remark \ref{rmk:integrablebna}).
	In this case, $\om$ is $q_{Ker(\om)}$-basic and $q_{Ker(\om) *}\om$ has trivial kernel. 
\end{proposition}
\begin{proof}
	The implication $\Rightarrow$ follows from Lemma \ref{lem:keromlin}. For the converse, denote $q$ the simple quotient of $E\to M$ defined by the quotient data $(q_M,K=Ker(\mu), \Delta)$. By Lemma \ref{lem:keromlin} and the definition of being kernel-reducible, we only need to show that $Ker(Tq)=Ker(\om)$.
	
	Around each point in $E$, we can use the adapted-coordinate description of $Ker(\om|_{E|_U})$ used in the proof of Lemma \ref{lem:keromlin}, yielding
	$$ Ker(\om|_{E|_U}) = \{ (\dot x, \dot u)\in T_{x}U \times \R^{rk(E)}: \dot x \in Ker(\mu^\sharp), \ \dot u^\a=0  \},$$
	where the fiber coordinates $u^a$ are defined using local basis of sections $\{e_a\}$ of $E|_U$ for which the first $rk(Ker(\mu))$-elements span $Ker(\mu)$ and the rest, which are denoted $\{e_\a\}$ and induce a basis of $E/Ker(\mu)$, are taken to be $\bna$-flat. In these coordinates, $K=Ker(\mu)$ is defined by $u^\a=0, \forall \a$.  

	Using the same set of coordinates, recalling that $Ker(\mu^\sharp)=Ker(Tq_M)$ by hypothesis, it follows from the coordinate description given in Remark \ref{rmk:kerTqcoord} for $Ker(Tq)$, that $Ker(\om)=Ker(Tq)$ as wanted.
\end{proof}

For the reader's convenience, we summarize the conditions for $\om$ to be kernel-reducible in a non-redundant list.

\begin{corollary}\label{cor:linomkerred}
	Let $\om$ be a linear $k$-form on $E\to M$ with components $(\mu,\eta)$. Then, $\om$ is kernel-reducible iff the following hold:
	\begin{enumerate}
		\item $\mu:E \to \Lambda^{k-1}T^*M$ has constant rank as a vector bundle morphism over $id_M$
		\item there exists a simple quotient $q_M: M \to \tM$ such that $Ker(\mus)=Ker(Tq_M)$
		\item $Ker(\mu)\subset Ker(\eta)$ and $Ker(Tq_M)\subset Ker(\eta^\sharp)$
		\item  there exists a representation $\Delta$ of $M\times_{\F(q_M)} M$ on $E/Ker(\mu)$ (see Remark \ref{rmk:integrablebna}) such that
		$$  i_{v} d\mu(e) = \mu\left(\underset{=:\bna_v\bar e|_{x(0)}}{ \underbrace{\frac{d}{d\epsilon}|_{\epsilon=0} \Delta_{x(0),x(\epsilon)} \bar e( x(\epsilon))}} \right) \text{ where $v=\frac{d}{d\epsilon}|_{\epsilon=0}x(\epsilon)$}$$
		for all $e\in \Gamma E$ and $v\in Ker(Tq_M)$, where $\bar e := e \ mod \ Ker(\mu)$ is the induced section of $E/Ker(\mu)$.
		\item if $\bna_v \bar e =0$ then $i_v d\eta(e) =0$.
	\end{enumerate}
\end{corollary}
Notice that, when they exist, both $q_M$ and $\Delta$ are fully determined by $\om$, hence they are unique.

We now consider a Lie algebroid structure $A$ on $E \to M$ and study when it is basic for the kernel of $\om$. 
\begin{theorem}\label{thm:omActrk}
	Let $A$ be a Lie algebroid structure on $E\to M$ and $\om$ a linear $k$-form on $E\to M$, $k\geq 2$. If $\om$ is $A$-morphic and $\om$ is kernel-reducible with underlying simple quotient $q=q_{Ker(\om)}$, then $A$ is $q$-basic. In particular, $q_*\om$ defines a $q_*A$-morphic $k$-form with trivial kernel: denoting $(\tmu,\teta)$ the components of $q_*\om$, defined by eqs. \eqref{eq:tildemuetadef} in terms of the components of $\om$, then
	$$ Ker(\tmu)=0, \ Ker(\tmu^\sharp)=0 \text{ and $(\tmu,\teta)$ satisfy the "IM" equations \eqref{eq:IMeqs} for $q_*A$.}$$
\end{theorem}

\begin{proof}
	We need to verify that $A$ and $q$ satisfy conditions (1.-3.) in Proposition \ref{prop:Abasic}, saying $A$ is $q$-basic, as a consequence of $\om$ being $A$-morphic. We recall that $\om$ being kernel-reducible implies that the quotient data $(q_M,K,\Delta)$ underlying $q=q_{Ker(\om)}$ satisfies 
	$$Ker(Tq_M)=Ker(\mu^\sharp), \ K=Ker(\mu), \ \Delta \text{ integrates } \bna,$$
	where $\bna$ is defined in Lemma \ref{lem:keromlin}. Moreover, from Proposition \ref{prop:qomred}, we know that $\om$ is $q$-basic and we denote $(\tmu,\teta)$ the components of $q_*\om$.
	The fact that $q_*\om$ has trivial kernel is general (see Section \ref{subsec:general}), from which we obtain $Ker(\tmu)=0$ and $Ker(\tmu^\sharp)=0$.
	 
	The first point (1.) in order to show $A$ is $q$-basic is equivalent to:
	$$ \text{(1a:) } \rho(K)\subset Ker(Tq_M) \text{ and (1b:) } \rho(e) \text{ is $q_M^\nu$-invariant for each $q$-invariant $e\in \Gamma(E)$}.$$
	(Recall the definition of $q_M^\nu$ from Example \ref{ex:normalquot}.)
	For $k\in K_x, \ e\in E_x$, then 
	$$\mu^\sharp(\rho(k))(e)=i_{\rho(k)}\mu(e) \overset{\eqref{eq:IMeqs}}{=} -i_{\rho(e)}\mu(k)=0.$$
	This shows (1a.): $\rho(K)\subset Ker(\mus)=Ker(Tq_M)$. 
	
	Item (1b.) requires more refined examination of the underlying geometry.
	Around each point in $M$, we can take a foliated chart $U=F\times \tU$ so that $q_M$ is the projection onto $\tU$. In this chart, for a $q$-invariant section $e_1$, we write
	\begin{equation}\label{eq:rirho}Tq_M\rho(e_1|_x) = r^i(x) \tilde v_i, \text{ for $v_i \in \mathfrak{X}(\tU)$ a local basis and $r^i\in C^\infty(U)$.}\end{equation}
	With this notation, (1b.) is equivalent to the functions $r^i$ being independent of the $F$-coordinates: $d_Fr^i=0 \forall i$.
	We now characterize $q$-invariant sections further. Denoting $(q,q_M): (E\to M) \to (\tilde E\to \tM)$ the simple quotient defined by the kernel of $\om$, we choose a bundle isomorphism
	$$ E \simeq K \oplus_M q_M^*\tilde E,$$
	under which $q$-invariant sections must be of the form 
	$$e=k+q_M^*\te \text{ for $k\in \Gamma_M(K)$ and $\te = q_*e \in \Gamma_{\tM}(\tE)$.}$$
	For such a $q$-invariant section $e$, since $\om$ is kernel-reducible, 
	$$\mu(e) = q_M^*\tmu(\te), \ \eta(e)=q_M^*\teta(\te).$$
	
	We now take $e_1,e_2$ $q$-invariant sections, and use the second IM equation in \eqref{eq:IMeqs} for the components $(\mu,\eta)$ of an $A$-morphic form, to get
	\begin{equation} \label{eq:muonbrakinvar} \mu([e_1,e_2]) = d i_{\rho(q^*_M\te_1)} q_M^*\tmu(\te_2) + Z, \text{ where } Z \in C^\infty(M)\otimes q_M^*\Omega(\tM),\end{equation}
	and we have used $\rho(K)\subset Ker(\mus)\subset Ker(\eta^\sharp)$. We can write
	\begin{equation} \label{eq:qinvarbrack} [e_1,e_2]|_x = c^a(x) e_a|_x, \text{ for $e_a$ being invariant sections.}\end{equation}
	(Every section is locally a $C^\infty(M)$-linear combination of $q$-invariant sections, see Remark \ref{rmk:localinvarbasis}.) Restricting to a foliated chart $U\subset M$, this expression and eq. \eqref{eq:rirho} imply
	$$ d_Fr^i \wedge q_M^*i_{v^i}\tmu(\te_2) + Z'=0, \text{ where $Z'\in C^\infty(U)\otimes q_M^*\Omega(\tU)$} .$$
	The first term above is the only one with a factor on $\Omega^1(F)$, so it should vanish independently. We also know that the above equation holds for every $q$-invariant section $e_2$ and that $Ker(\tmu^\sharp)=0$. It follows (e.g: using local coordinates on $F$) that $d_Fr^i=0 \ \forall i$, thus proving (1b.). 
	
	The item (2.) in Proposition \ref{prop:Abasic} reads:
	$$ \text{if $e$ is $q$-invariant and $k\in \Gamma(Ker(\mu))$ then $\mu([k,e])=0$. } $$
	Using the second IM equation in \eqref{eq:IMeqs}, we compute
	$$ \mu([k,e])= L_{\rho(k)}q_M^*\tmu(\te) + 0,$$
	where we have used $Ker(\mu)\subset Ker(\eta)$. Now, the r.h.s. is zero since $\rho(k) \in \Gamma(Ker(Tq_M))$ by (1a.), thus proving (2.).
	
	Finally, item (3.) in Proposition \ref{prop:Abasic} reads:
	$$ \text{if $e_1,e_2$ are $q$-invariant, then $[e_1,e_2]$ is $q$-invariant. } $$
	Using: a foliated chart $U=F\times \tU$ as above, the second IM equation in \eqref{eq:IMeqs} once more and items (1a.,1b.) proven above, we get
	$$ \mu([e_1,e_2]) \in q_M^*\Omega(\tU),$$
	which shows that $[e_1,e_2]$ is $q$-basic, as wanted. The fact that $A$ is $q$-basic is thus proven.
	
	Now that we know that $A$ is $q$-basic, since $\om$ is also $q$-basic by hypothesis, Proposition \ref{thm:Aandomqbasic} implies that $q_*\om$ is morphic for $q_*A$. 
\end{proof}

\section{Quotients of multiplicative forms on Lie groupoids and their Lie theory}\label{sec:mult}

We state some conventions to be used through this section. For a Lie groupoid $G\rra M$ we denote by $s$ its source map, $\t$ its target map, $\inv(g)=g^{-1}$ the inversion map and $1:M\to G$ the unit map. The multiplication map is denoted $m(g,h)=gh$ and defined when $s(g)=\t(h)$.

We denote $A=Lie(G)$ the associated Lie algebroid structure on the vector bundle

$$E_G:=Ker(Ts|_{1(M)})\to M,$$ where we identify sections of $E_G$ with right invariant vector fields on $G$. The anchor map of $Lie(G)$ is given by $\rho(e)=T\t(e)$ for $e\in E_G$.

Given a Lie groupoid morphism $$(F,f): (G\rra M) \to (\tilde G \rra \tM)$$
we denote the associated vector bundle morphism as
$$ (Lie(F),f): (E_G \to M) \to (E_{\tilde G} \to \tM), \ Lie(F)(e) = TF(e), \ e \in Ker(Ts|_{1(M)}) \subset TG|_{1(M)},$$
which defines a Lie algebroid morphism $Lie(G)\to Lie(\tilde G)$.

\subsubsection*{A preliminary technical result}
Before beginning our study of quotients of multiplicative forms, we briefly describe a technical result that will be used in the main results.

\begin{lemma}\label{lma:factsgds}
	Let $G\rra M$ be a Lie groupoid and $F: (G\rra M) \to (\R \rra \star)$ be a groupoid $1$-cocycle seen as a morphism onto the Lie group $\R$ over a point $\star$. Denote  $\xi=Lie(F): (E_G \to M) \to (\R\to \star)$ the induced algebroid $1$-cocycle.
	\begin{enumerate}
		\item for any smooth curve $g:[0,1]\to G$ such that $s(g(t))$ is constant then,
		$$ F(g(t))=F(g(0))+\int_0^t \xi(TR_{g^{-1}(s)}\frac{d}{dt}g(s)) \ ds.$$
		\item if $(R\rra N) \subseteq (G\rra M)$ a Lie subgroupoid with connected source fibers and $Lie(F)|_{E_R}=0$ then $F|_R =0$.
	\end{enumerate}
\end{lemma}

\begin{proof}
	For (1.), we note that $F(g(t+\epsilon))=F(g(t+\epsilon)g(t)^{-1}g(t))=F(g(t+\epsilon)g(t)^{-1}) + F(g(t))$. Taking $d/d\epsilon$, using the definition of $\xi=Lie(F)$ and the fundamental theorem of calculus for $t\mapsto F(g(t))$, we obtain the desired formula.
	For (2.), let $r\in R$. Since $R$ is source-connected, there exists a smooth path $r:[0,1]\to R$ such that $r(1)=r$, $r(0)=1(s(r))$ and $s(r(t))$ is constant. Notice that, since $R\subseteq G$ is a Lie subgroupoid, $TR_{r(t)^{-1}}\frac{d}{dt}r (t) \in (E_R\to N)\subset (E_G\to M)$ for all $t$. Using item (1.), we obtain
	$$ F(r(1))= F(r(0)) + \int_0^1 \xi(TR_{r(t)^{-1}}\frac{d}{dt}r (t)) \ dt = 0,$$
	where the first term vanishes for being a $1$-cocycle evaluated at a unit and the second by the hypothesis $Lie(F)|_{Lie(R)}=0$.
\end{proof}

\subsection{Simple quotients of Lie groupoids}\label{subsec:simpleG}

We begin with the definition of a simple quotient which takes into account the Lie groupoid structure.
\begin{definition}\label{def:simpquotG}
	A {\bf simple quotient of a Lie groupoid} $G\rra M$ is a Lie groupoid morphism
	$$ (\q,q_M): (G\rra M) \to (\tilde G\rra \tilde M)$$
	such that:
	\begin{enumerate}
		\item $q_M$ is a surjective submersion, 
		\item the map
			$$G \to \tG\ _{\tilde s}\times_{q_M} M=\{(\tilde g, x)\in \tG\times M:\tilde s(\tilde g)=q_M(x) \}, \ g\mapsto (\q(g),s(q))$$
			is a surjective submersion,
		\item both the $q_M$-fibers and the $\q$-fibers are connected.
	\end{enumerate}
\end{definition}
Conditions \emph{1.} and \emph{2.} above say that $\q:G\to \tilde{G}$ is a {\bf Lie groupoid fibration} as in \cite[Def. 2.4.3]{mackenziebook}. Condition \emph{3.} is added in order to characterize basic differential forms on $G$ in terms of infinitesimal data, as in Section \ref{subsec:general}. Notice that, since $q_M$ is a surjective submersion, then the projection $\tG\ _{\tilde s}\times_{q_M} M \to \tG$ is a surjective submersion. Thus, $\q$ is also a surjective submersion. Also notice that condition (2.) is stronger than requiring $\q$ to be a surjective submersion and that it generalizes the "fiberwise surjective" condition appearing for simple quotients of vector bundles (see Section \ref{subsec:simpleqvec}).
We sometimes denote the quotient Lie groupoid by
$$ \tilde G = \q_* G.$$

\begin{remark} Just as a simple quotient of a vector bundle is completely determined by its components, a simple quotient $\q$ of a Lie groupoid $G\rra M$ is also characterized by components $(q_M,\K,\theta)$ which define a \emph{normal subgroupoid system} as introduced in \cite[Def. 2.4.7]{mackenziebook}. Nevertheless, these components will not play a role in the paper so we will not go into the details of this aspect of simple quotients of Lie groupoids.
\end{remark}

\medskip

{\bf Lie theory of simple quotients.}
Given a simple quotient $\q$ of $G\rra M$, then the induced vector bundle map $Lie(\q):E_G \to E_{\tilde G}$ defines a simple quotient of $(E_G\to M)$, the Lie algebroid structure $Lie(G)$ on $E_G$ is $Lie(\q)$-basic and $Lie(\q)_*Lie(G) = Lie(\tilde G)$ (see \cite[Thm. 4.4.7]{mackenziebook}).
Conversely, if $q$ defines a simple quotient of a Lie algebroid $A$, we say that $q$ is {\bf integrable to a simple quotient} if there exist a Lie groupoid $G\rra M$ and a simple quotient $\q$ of $G\rra M$ such that 
$$ Lie(G)\simeq A, \ Lie(\q)\simeq q.$$
In this case, we say that $(G\rra M,\q)$ yields a {\bf simple quotient integration} of $(Lie(G),q_A)$.

%
Given a simple quotient $q$ of a Lie algebroid $A$ on $E\to M$, the problem of finding an integration $G\rra M$ of $A$ and a simple quotient $\q$ of $G\rra M$ such that $Lie(\q)=q$ is non-trivial in general. As far as the authors know, it has to be treated case-by-case. 



In the rest of the paper, we will consider instances in which the integration $\q$ of $q$ can be found, and leave the underlying integrability problem for separate investigations.

\medskip

In the special case of Lie algebroid simple quotients of pullback type, there exists a general criterion which we summarize in the following.
\begin{example}[Integrating simple quotients of pullback type]\label{ex:intqpullbacktype}

Consider a simple quotient $q:A\to \tA$ of a Lie algebroid $A$ which is of pullback type, as in Remark \ref{rmks:pullbackq}, so that $A\simeq q_M^!\tA$ is isomorphic to the pullback algebroid. Some results in \cite{Alv} can be used to provide a criterion for $q$ to admit an integration to a simple quotient of a Lie groupoid, as follows. Recall that, since $q$ is of pullback type, there is a Lie algebroid injection $I_A\equiv (\rho|_K)^{-1} : Ker(Tq_M)\to A$ over $id_M$, where the involutive distribution $Ker(Tq_M)\subset TM$ is endowed with its natural algebroid structure. Assume that there exists a Lie groupoid morphism
	$$ I: (M \times_{\F(q_M)} M\rra M) \to (G\rra M), \text{ such that $Lie(I)= I_A$}$$
	(in particular, $Lie(G)= A \simeq q_M^!\tA$). There is an induced action of the direct product $(M \times_{\F(q_M)} M\rra M) \times (M \times_{\F(q_M)} M\rra M)$ on $G$ given by
	$$ (z_1,z_2)\cdot g = I(z_1)gI(z_2), \ z_1,z_2 \in M \times_{\F(q_M)} M, \ g\in G$$
	where the obvious composability conditions are implicit. Denote $\F\subset G$ the foliation given by the orbits of this, and
	$$ \q: G \to G/\F$$
	the quotient to the orbit space. Then, by \cite[Thm. 3.2 and its proof]{Alv}, there is a unique Lie groupoid structure on $G/\F$ making $\q$ a simple quotient of $G\rra M$. Note that the $\q$-fibers are connected since the $q_M$-fibers are. Moreover, $Lie(G/\F)\simeq \tA$ and  $Lie(\q)\simeq q_A$. 
\end{example}

\begin{remark}[Infinitesimal ideal system underlying $\q$]
	Given a simple quotient $\q$ of $G\rra M$, the simple foliation $\F(\q)$ given by the $\q$-fibers is multiplicative, i.e. $T\F(\q)\subseteq TG$ is a Lie subgroupoid of the tangent groupoid. For any regular multiplicative distribution $\F$ on $G$, \cite{JO} shows how to associate an \emph{infinitesimal ideal system}  $(\F_M,K,\bna)$ on the Lie algebroid $A=Lie(G)$, where $\F_M$ is a regular foliation of $M$. For the foliation $\F(\q)$ defined by the simple quotient $\q$ above, it is easy to verify that the corresponding infinitesimal ideal system has $\F_M=\F(q_M)$, $K=Ker(T\q)\cap E_G$ and $\bna$ given by the differentiation (see \ref{rmk:integrablebna}) of the component $\Delta$ underlying the simple quotient $Lie(\q)$ of $E_G$. When $G$ is source-simply-connected, the components $(q_M,K,\bna)$ underlying $\q$ determine the foliation $\F(q)$ uniquely (see \cite[Cor. 4.10]{JO}).
\end{remark}

\begin{remark}[Quotient by a multiplicative foliation]
	If $\F$ is a multiplicative foliation of $G\rra M$, as in \cite{J,JO}, then the set-theoretic quotient $\q:G \to G/\F$ naturally induces on $G/\F$ source, target, units and inversion map so that $\q$ commutes with the corresponding maps on $G$. Nevertheless, there might be an obstruction to finding a compatible multiplication map on $G/\F$ so that $\q$ defines a groupoid morphism, see \cite{J}. Thus, even when $\F$ is simple (as a foliation), there might be an obstruction for $G/\F$ to admit a Lie groupoid structure so that $\q$ is a morphism.
\end{remark}

\subsubsection*{Basic multiplicative forms on Lie groupoids}

A $k$-form $\gom\in \Omega^k(G)$ is called {\bf multiplicative } if
$$ m^*\gom = pr_1^*\gom + pr_2^*\gom,$$
where $pr_j(g_1,g_2)=g_j, \ j=1,2, \ (g_1,g_2)\in G^{(2)}.$ Following \cite{bursztyncabrera}, one observes that the multiplicativity of $\gom\in \Omega^k(G)$ is equivalent to the induced map
$$ \underline{\om}: T^{\oplus k}G:=TG\oplus_G \dots \oplus_G TG \to \R,$$
being a Lie groupoid morphism from the direct sum of the tangent Lie groupoid structures $TG\rra TM$ into the abelian group $\R\rra \star$.

We now verify that being multiplicative is preserved under simple quotients of $G$.
\begin{lemma}\label{lem:quotommult}
	Let $(\q,q_M)$ be a simple quotient of the Lie groupoid $G\rra M$ and $\om \in \Omega^k(G)$. If $\om$ is $\q$-basic, then $\om$ is multiplicative on $G$ iff $\q_*\om$ is multiplicative on $\q_*G$.
\end{lemma}
\begin{proof}
	Denote $\tom=\q_*\om$ and $\tG=\q_*G$. Since $\q: G \to \tG$ is a surjective submersion, $\tom$ is multiplicative iff
	$$ \q^*(\tilde{m}^*\tom - \tilde{pr}_1^*\tom - \tilde{pr}_2^*\tom) = 0,$$
	with the obvious notations for the structure maps associated to $\tG\rra \tM$. Using that $\q$ is a Lie groupoid morphism and $\q^*\tom = \om$, the above is equivalent to $\om$ being multiplicative.
\end{proof}

\subsection{Lie theory for basic morphic forms}
Let $G\rra M$ be a Lie groupoid and $A=Lie(G)$ the associated Lie algebroid structure on $E_G=Ker(Ts|_{1(M)})\to M$.
We recall from \cite{bursztyncabrera} that the Lie functor induces an assignment
$$\Omega^k(G) \ni \om \mapsto Lie(\om) \in \Omega^k(E_G)$$
which takes multiplicative $k$-forms on $G\rra M$ to $A$-morphic linear $k$-forms on $E_G$. Denoting $(\mu,\eta)$ the components of $Lie(\om)$, we say that {\bf $(G,\om)$ integrates $(A,(\mu,\eta))$} and that $(\mu,\eta)$ are the {\bf infinitesimal components} of $\om$. 
Following \cite{bursztyncabrera}, the infinitesimal components $(\mu,\eta)$ can be defined directly in terms of $\om$ by the formulas:
\begin{eqnarray}
\mu(e)(v_1,\dots, v_{k-1}) &=& \om|_{1(x)} (e,T1(v_1),\dots, T1(v_{k-1})), \nonumber \\
\eta(e)(v_1,\dots, v_{k}) &=& d\om|_{1(x)} (e,T1(v_1),\dots, T1(v_{k}) ) \label{eq:muetfrommult},
\end{eqnarray}
where $e\in E_x$ and $v_1,\dots, v_k \in T_xM$, $x\in M$.

Our study below is guided by trying to produce integrations of infinitesimal components as a result of quotient procedures.

\begin{proposition}\label{prop:Liequotcomm}
	Let $\q$ be a simple quotient of a Lie groupoid $G\rra M$ and $\om$ a multiplicative $k$-form on $G\rra M$. If $\om$ is $\q$-basic, then $Lie(\om)$ is $Lie(\q)$-basic and 
	$$ Lie(\q_*\om)=[Lie(\q)]_*(Lie(\om))$$
	In particular, $(\q_* G, \q_* \om)$ integrates $(Lie(\q)_*A, (\tmu,\teta))$ where $(\tmu,\teta)$ are the components of $\tom_A$.
\end{proposition}

\begin{proof}
	Following \cite{bursztyncabrera}, the infinitesimal form $Lie(\om)\in \Omega^k(E_G)$ can be computed as
	$$ Lie(\om)= j^*(\om_T), \text{ where $\om_T \in \Omega^k(TG)$ is the \emph{tangent lift} of $\om$},$$
	and $j:E_G \hookrightarrow TG$ denotes the inclusion.
	
	Let us denote $\q:G\to \tilde G$ the quotient, $\tom = \q_* \om$ and $\tilde j: E_{\tilde G}=Lie(\q)_*E_G \hookrightarrow T\tilde G$ the inclusion. Then,
	$$ Lie(\om)= j^*(\om_T) = j^*((\q^*\tom)_T) = j^*((T\q)^*(\tom_T)) = Lie(\q)^*\tilde{j}^*(\tom_T) = Lie(\q)^*Lie(\tom),$$
	where we used: a general property of tangent lift $(f^*\om)_T = (Tf)^*(\om_T)$; the definition of $Lie(\q)$, $T\q\circ j = \tilde{j} \circ Lie(\q)$; and the previous characterization of $Lie$ of a multiplicative form. This proves the Proposition.
\end{proof}

We now address the problem of giving infinitesimal characterization for the hypothesis needed to obtain quotients.

\begin{proposition}\label{prop:omqbasicfrominfi}
Let $\q$ be a simple quotient of $G\rra M$ and $\om$ be a multiplicative form on $G$. If $G$ is source-connected and $Lie(\om)$ is $Lie(\q)$-basic, then $\om$ is $\q$-basic.
\end{proposition}
\begin{proof}
	The key idea is to observe that $R_k:=Ker(T\q)\oplus_G TG \oplus_G \dots \oplus_G TG \subset T^{\oplus k} G$ defines a Lie subgroupoid for each $k$. Since $G\rra M$ is source-connected, then both $R_k$ and $T^{\oplus k}G\rra T^{\oplus k}M$ are source-connected as well. Both $\underline{\om}$ and $\underline{d\om}$ are Lie groupoid $1$-cocycles and $\om$ is $\q$-basic iff $$ \underline{\om}|_{R_k} = 0, \ \underline{d\om}|_{R_{k+1}} = 0.$$
	We apply Lemma \ref{lma:factsgds}(2.) and conclude that the above holds when it holds infinitesimally: 
	$$\underline{Lie(\om)}|_{Lie(R_k)}=0, \ \underline{Lie(\om)}|_{Lie(R_{k+1})}=0.$$
	But this is equivalent to the hypothesis of $Lie(\om)$ being $Lie(\q)$-basic.
\end{proof}

We summarize the discussion in the following theorem.
\begin{theorem}
	Let $G\rra M$ be a Lie groupoid with connected source fibers, $\om$ a multiplicative $k$-from on $G$ and $\q$ a simple quotient of $G\rra M$. If $Lie(\om)$ is $Lie(\q)$-basic, then $\om$ is $\q$-basic and
	$$ \text{$(\q_*G,\q_*\om)$ integrates $(\tilde A, (\tmu,\teta))$,}$$
	where $Lie(\q):Lie(G) \to \tilde A$ is the induced simple quotient of $Lie(G)$ and $(\tmu,\teta)$ are the components of $Lie(\q)_*Lie(\om)$.
\end{theorem}

As mentioned in Section \ref{subsec:simpleG}, given the infinitesimal data $(A,q_A,(\mu,\eta))$, it is in general a case-by-case study to find corresponding integrations $(G\rra M, \q, \om)$ such that $\q$ defines a simple quotient of $G\rra M$.

\subsection{Quotient by the kernel of multiplicative forms and their Lie theory}
To begin with, following Section \ref{subsec:general}, we say that a multiplicative form $\om$ on $G\rra M$ is {\bf kernel-reducible} when:
\begin{enumerate}
	\item $Ker(\om)\subseteq TG$ has constant rank as a distribution on $G$,
	\item $Ker(\om)\subseteq Ker(d\om)$,
	\item the quotient map $\q:=\q_{Ker(\om)}:G\to G/\F_{Ker(\om)}$ into the leaf space of the foliation $\F_{Ker(\om)}$ integrating $Ker(\om)\subseteq TG$ defines a simple quotient of $G\rra M$.
\end{enumerate}
Notice that, as in the general case discussed in Section \ref{subsec:general}, the foliation $\F(\q)$ defined by the $\q$-fibers is uniquely determined by $\om$. 
When a multiplicative form $\om$ is kernel reducible, then $\q_*\om$ defines a multiplicative form (see Lemma \ref{lem:quotommult}) with trivial kernel. We anticipate that, using Proposition \ref{prop:Liequotcomm} and Lemma \ref{lem:Liekerred} below, if $\om$ is kernel-reducible, then $Lie(\om)$ is kernel-reducible (see Section \ref{subsec:ctrkA}) and $\q_*\om$ integrates the IM form obtained as a quotient of $Lie(\om)$ by $Ker(Lie(\om))$.


In the remaining of this subsection, we proceed with a Lie-theoretic analysis of conditions 1. and 2. for kernel reducibility. We will always denote $(\mu,\eta)$ the components of a multiplicative  form $\om$ as defined in \eqref{eq:muetfrommult}. The following Lemma states some basic facts about the kernel of multiplicative forms.
\begin{lemma}\label{lma:basicmultprop}
	Let $\om$ be a multiplicative $k$-form on $G\rra M$. 
	\begin{enumerate}
		\item for every $x\in M$, the natural decomposition $T_{1(x)}G= E_G|_x \oplus D_x1(T_xM)$ gives rise to a splitting $Ker(\om|_{1(x)}) = Ker(\mu_x)\oplus D_x1(Ker(\mus_x))$,
		\item if $V_g \in Ker(\om|_g)$, $W_h \in Ker(\om_h)$ and $D_gs(V_g)=D_h\t(W_h)$ then 
		$$V_g\cdot W_h:= D_{(g,h)}m(V_g,W_h)\in Ker(\om|_{m(g,h)})$$
		\item if $V_g \in Ker(\om|_g)$ then $D_g\inv(V_g) \in Ker(\om|_{g^{-1}})$. In particular, if $V\in Ker(\om|_g)$ then $Ds(V)\in Ker(\mus_{s(g)})$ and $D\t(V) \in Ker(\mus_{\t(g)})$.
		\item $V  \in Ker(\om|_g)$ iff $D_g \t(V)\in Ker(\mus_{\t(g)})$ and
		$$ \om(V,\tilde v_2,\dots, \tilde v_k) =0 \text{ for all $\tilde v_2,\dots, \tilde v_k$ in a complement of $Ker(D_gs)\subset T_g G$.}$$
	\end{enumerate}
\end{lemma}
These properties follow directly from the multiplicativity of $\om$ combined with the fact that the multiplication map on $G$ is a surjective submersion.
The following lemma says that differentiation preserves kernel-reducibility.
\begin{lemma}\label{lem:Liekerred}
If $\om\in\Omega^k(G)$ is a kernel-reducible multiplicative form inducing a simple quotient $\q_{Ker(\om)}:G\to G/\F_{Ker(\om)}$, then $Lie(\om)\in\Omega^k(E_G)$ is kernel-reducible (as in Section \ref{subsec:ctrkA}). In this case, the quotient maps are related by $Lie(\q_{Ker(\om)}) = q_{Ker(Lie(\om))}$.
\end{lemma}
This Lemma is a direct consequence of the definitions involved.
We now go in the opposite direction of "integration". The following proposition gives an infinitesimal criterion for a multiplicative form to have constant rank, which might be of independent interest.

\begin{proposition} \label{prop:ctrankfrominfinit}
	Assume that $G\rra M$ has connected source fibers and let $\om$ be a multiplicative $k$-form. If $Ker(Lie(\om))$ has constant rank and $Ker(Lie(\om))\subset Ker(dLie(\om))$ then $Ker(\om)$ has constant rank and $Ker(\om)\subset Ker(d\om)$.
\end{proposition}
\begin{proof}
	Let us denote $\om_A = Lie(\om)$ with components $(\mu,\eta)$ which, by hypothesis, must satisfy all properties in Lemma \ref{lem:keromlin}. By Lemma \ref{lma:basicmultprop}, we know that the rank of $Ker(\om|_1)$ equals the rank of $Ker(\om_A)$. We thus first need to show that $dim \ Ker(\om|_g) = rk(Ker(\om_A))$ for any $g\in G$. 
	
	To that end, we first discuss how the hypothesis reflect on the local geometry of the underlying bundle $E\equiv E_G$. The fact that $\om_A$ has constant rank and $Ker(\om_A)\subset Ker(d\om_A)$ implies (via the equivalences of Lemma \ref{lem:keromlin}) to the fact that $Ker(\om_A)$ defines locally a simple quotient of $E\to M$, i.e. for each $x\in M$, there exists an open $U\ni x$ in $M$ and a simple quotient $q_U$ of $E_U:=E|_U\to U$ such that $Ker(Tq_U)=Ker(\om_A|_{E_U})$. The restriction of $\om_A$ to such $E_U$ becomes $q_U$-basic. Within this proof, when $U$ satisfies the above and $E_U$ admits a basis of $q_U$-invariant sections, we say that $U$ is an \emph{admissible} open set. It is clear that $M$ admits a covering by admissible open sets (see Remark \ref{rmk:localinvarbasis}).
	
	Now, let us fix $g\in G$. Since $G$ is source-connected, following \cite[Lemma 2.3]{MX} there exists a finite collection of compactly supported sections $e_1,\dots, e_n \in \Gamma_M E_G$ such that
	\begin{equation}\label{eq:gpathid} g = \phi_{t=1}^{e_n^R}\circ \dots \circ \phi_{t=1}^{e_1^R} (1_x), \ x=s(g),\end{equation}
	where $e^R \in \mathfrak{X}(G)$ denotes the right-invariant vector field induced  by a section $e\in \Gamma_M E_G$. Moreover, at the possible cost of enlarging $n$, we shall see that the sections $e_j$ can be chosen in a special way. First, the support of $e_j$ can be taken to lie inside an admissible open set $U_j\subset M$. Second, recursively denoting 
	$$g_j(t)= \phi_{t}^{e_j^R}(g_{j-1}(1)), \ g_1(t)=\phi_{t}^{e_1^R}(1_x)$$ 
	we can choose each $e_j$ so that $\t(g_j(t)), \ t\in [0,1]$ lies in an open $V_j\subset U_j$ such that $e_j|_{V_j}$ is $q_{U_j}$-invariant. This is possible since $E_{U_j}$ admits a basis of $q_{U_j}$-invariant sections and $s$ is locally a product. Within this proof, we say that each such $e_j$ is \emph{admissible relative to} $g_{j-1}$".

	The key idea will be that elements of the kernel of $\om$ can be transported along each piece of the above broken path.
	\begin{lemma}
		For any $e\in \Gamma_M E_G$ admissible relative to $g_0\in G$ as above, denote $g(t)= \phi_{t}^{e^R}(g_0)$. Consider a decomposition $T_{g_0}G=Ker(D_{g_0}s)\oplus C_{g_0}$ and $V_0 = b^R|_{g_0} + \tilde v \in Ker(\om|_{g_0})$, with $b\in \Gamma_M Ker(\mu)\subset \Gamma_M E_G$ and $\tv \in C_{g_0}$. Then, $V_t:=b^R|_{g(t)} + D_{g_0}\phi_{t}^{e^R}\tv \in Ker(\om|_{g(t)})$ for all $t\in [0,1]$.
	\end{lemma}
	\begin{proof}(of the Lemma)
		We first notice that $D_{g_0}\phi_{t}^{e^R}(C_{g_0})\subset T_{g(t)}G$ defines a complement for $Ker(D_{g(t)}s)\subset T_{g(t)}G$. 	According to Lemma \ref{lma:basicmultprop}, we need to check that 
		$$ \text{(1.) } D_{g(t)}\t(V_t) \in Ker(\mus_{\tau(g(t))}),\text{ and (2.) } \om_{g(t)}(V_t,D_{g_0}\phi_{t}^{e^R}(\tv_2),\dots D_{g_0}\phi_{t}^{e^R}(\tv_k))=0,$$
		for all $\tv_2,\dots, \tv_k \in C_{g_0}$ and all $t\in [0,1]$. We also recall that the right-invariant vector field $e^R$ is projectable along the target map $\t$ to $\rho(e) \in \mathfrak{X}(M)$, so that 
		$$ x(t):=\t(g(t))=\phi^{\rho(e)}_t(\t(g_0)).$$
		
		Now, using $\t(gh)=\t(g)$, the definition of the anchor $\rho$ and the definition of $V_t$, (1.) above reads
		$$ \rho(b|_{x(t)})+ D_{x_0}(\phi^{\rho(e)}_t\circ \t)(\tv) \in Ker(\mus_{x(t)}),\ \forall t\in [0,1].$$
		Since $(\mu,\eta)$ satisfy the IM equations \eqref{eq:IMeqs}, we know that $\rho(Ker(\mu))\subset Ker(\mus)$, so (1.) becomes equivalent to 
		$$v(t):=D_{x_0}(\phi^{\rho(e)}_t\circ \t)(\tv) \in Ker(\mus_{x(t)}),\ \forall t.$$
		Since $V_0 \in Ker(\om|_{g_0})$ by hypothesis, then $v(0)=D_{x_0}\t(\tv)\in Ker(\mus_{x(0)})$. The idea is that, when $e$ is $q_{Ker(\om_A)}$-invariant, then the linearized flow $T\phi^{\rho(e)}_t$ preserves $Ker(\mus)\subset TM$. Indeed, using that $Ker(\mus)$ is involutive by Lemma \ref{lem:keromlin}, it is enough to verify it for small $t$ and on an arbitrary adapted coordinate chart $U=F\times \tilde U\ni (y,\tilde x)$ in which $Ker(\mus)$ is given by the tangent to the $F$-fibers. From the proof of Theorem \ref{thm:omActrk} (see eq. \eqref{eq:rirho}), we know $\rho(e)$ is projectable to $\tilde U$, taking the form
		$$ \rho(e) = \alpha^l(y,\tilde x) \partial_{y^l} + \beta^k(\tilde x) \partial_{\tilde x^k}.$$
		The linearized flow $(\delta y, \delta \tilde x)\mapsto T\phi^{\rho(e)}_t(\delta y, \delta \tilde x)$ then preserves the tangent to the $F$-fibers, defined  by $\delta \tilde x =0$, as wanted. This shows that (1.) holds.
		
		For (2.), we observe that the definition of $V_t$ makes (2.) equivalent to
		$$R(t):= [(\phi^{\rho(e)}_t\circ \t)^* \mu(b) + i_{\tv}(\phi^{e^R}_t)^*\om]|_{C_{g_0}^{\times k-1}} = 0,$$
		where we used $b^R|_g = b(\t(g))\cdot 0_g$, $V=T(1\circ \t)(V) \cdot V$ for any $V\in TG$ and the multiplicativity of $\om$. We now evaluate
		$$ E(t):= \om|_{g(t)}(D\phi^{e^R}_t(\tv), D\phi^{e^R}_t(\tv_2),\dots, D\phi^{e^R}_t(\tv_k) )$$
		in terms of the infinitesimal $k$-form $\om_A$. For that, we take the direct sum groupoid $T^{\oplus k}G\rra T^{\oplus k}M$, and consider the curve
		$$ \hat g(t):= (D_{g_0}\phi^{e^R}_t(\tv),D_{g_0}\phi^{e^R}_t(\tv_2),\dots, D_{g_0}\phi^{e^R}_t(\tv_k)) \in T^{\oplus k} G.$$ 
		By right-invariance, it is clear that $(D_{g(t)}s\oplus \dots \oplus D_{g(t)}s)\hat g(t)$ is constant, so that the curve $\hat g(t)$ remains on a fixed source-fiber of $T^{\oplus k}G\rra T^{\oplus k}M$. Then, Lemma \ref{lma:factsgds}(1.) implies that
		$$ E(t)=  \om|_{g_0}(\tv, \tv_2,\dots, \tv_k ) + \int_0^t [(e\circ \phi^{\rho(e)}_s\circ \t)^*\om_A]|_{g_0} (\tv,\tv_2,\dots, \tv_k) \ ds,$$
		where one uses that $DR^{T^{\oplus k}G}_{\hat g^{-1}}\frac{d}{dt}\hat g (t) = \left( Te(D_{g_0}(\phi^{\rho(e)}_t\circ \t) \tv),\dots, Te(D_{g_0}(\phi^{\rho(e)}_t\circ \t) \tv_k) \right) \in T^{\oplus k}A$ under the identification\footnote{A section $e$ of $E$ produces $e^R\in \mathfrak{X}(G)$ and its tangent lift (see eg. \cite{bursztyncabrera}) yields $(e^{R})_T\in \mathfrak{X}(TG)$. On the other hand, $Te:TM\to TE$ defines a section for the Lie algebroid $TA$ and thus generates a right-invariant vector field $(Te)^{R_{TG}}\in \mathfrak{X}(TG)$ via $Lie(TG)\simeq TA$. These two vector fields on $TG$ coincide, see eg. \cite{MX2}. } $Lie(T^{\oplus k}G)\simeq T^{\oplus k} A$. 
		
		Now,
		$$R(0)(\tv_2,\dots,\tv_k)=\mu(b)|_{x(0)}(D \t (\tv_2),\dots, D \t(\tv_k)) + \om|_{g_0}(\tv, \tv_2,\dots, \tv_k ) = \om|_{g_0}(V_0,\tv_2,\dots,\tv_k)=0$$
		for all $\tv_j\in C_{g_0}$ since $V_0 \in Ker(\om|_{g_0})$ by hypothesis. On the other hand, using the above computation of $E(t)$,
		\begin{eqnarray}
		\frac{d}{dt}R (t) &=& [(\phi_t^{\rho(e)}\circ \t)^*L_{\rho(e)} \mu(b) +   i_{\tv} (e\circ \phi^{\rho(e)}_t\circ \t)^*\om_A]|_{C_{g_0}^{\times k-1}}  \nonumber \\
		&\overset{\text{def. $v(t)$, Remark \ref{rmk:omlinpulle}}}{=} & [(\phi_t^{\rho(e)}\circ \t)^*\left(L_{\rho(e)} \mu(b) +   i_{v(t)}(d\mu(e)+\eta(e) ) \right) ]|_{C_{g_0}^{\times k-1}}  \nonumber \\
		&\overset{\text{IM eqs. \eqref{eq:IMeqs},} v(t)\in Ker(\mus)\subset Ker(\eta^\sharp) }{=} & [(\phi_t^{\rho(e)}\circ \t)^*\left(\mu([e,b]) + i_{\rho(b)}(d\mu(e)+\eta(e)) +   i_{v(t)}d\mu(e) \right) ]|_{C_{g_0}^{\times k-1}}  \nonumber \\
		&\overset{b\in Ker(\mu), \ \rho(b)\in Ker(\mus)\subset Ker(\eta^\sharp) }{=} & [(\phi_t^{\rho(e)}\circ \t)^*\left(\mu([e,b]) + i_{\rho(b)+v(t)}d\mu(e) \right) ]|_{C_{g_0}^{\times k-1}}  \nonumber \\
		&\overset{\text{Prop. \ref{prop:Abasic}(2.) for $Ker(\om_A)$}}{=} & [(\phi_t^{\rho(e)}\circ \t)^*\left(i_{\rho(b)+v(t)}d\mu(e) \right) ]|_{C_{g_0}^{\times k-1}}  \nonumber \\
		&\overset{\text{Lemma \ref{lem:keromlin}(6.), $\rho(b)+v(t)\in Ker(\mus)$}}{=} & [(\phi_t^{\rho(e)}\circ \t)^*\left(\mu(\bna_{\rho(b)+v(t)}(\bar e)) \right) ]|_{C_{g_0}^{\times k-1}}  \nonumber \\
		&\overset{\text{$e$ is $q_{Ker(\om_A)}$-invariant}}{=} & 0. \nonumber 
		\end{eqnarray}
		Hence $R(t)=0$ for all $t$, saying that (2.) holds and thus finishing the proof of the Lemma.
		\end{proof}

In order to finish the proof of Proposition \ref{prop:ctrankfrominfinit}, we apply the previous Lemma to each piece of the broken path \eqref{eq:gpathid}, starting with $g_0=1_x$, $e=e_1$ and $b(x) + T1(v)\in Ker(\om|_{1(x)})$ with $b\in \Gamma_M(Ker \mu)$ and $v\in Ker(\mus_x)$. Then, the above lemma says that $V_1 =b^R|_{g_1 } + D_{x}(\phi_{t=1}^{e_1^R}\circ 1)v \in Ker(\om|_{g_1}) $ where $g_1=\phi^{e_1^R}_{t=1}(1_x)$. Iterating this procedure we obtain that $$ b^R|_{g} +  D_{x}(\phi_{t=1}^{e_n^R}\circ \dots \circ \phi_{t=1}^{e_1^R}\circ 1)v \in Ker(\om|_g) \text{ for all $b\in \Gamma_MKer(\mu)$ and $v\in Ker(\mus_x)$.}$$
This shows that $dim(Ker \om|_g)\geq rk(Ker(\mu))+rk(Ker(\mus))=rk(Ker(\om_A))$ for all $g\in G$. 

Regarding the proof of the other inequality, let us fix $V\in Ker(\om|_g)$. It follows from Lemma \ref{lma:basicmultprop} that $D_gs(V)\in Ker(\mus_x)$. Hence, applying the Lemma above iteratively along the broken path \ref{eq:gpathid}, with $b=0$ and the initial $\tv_0=T1(D_gs(V))$ we obtain an element $\tv \in Ker(\om|_g)$ satisfying $D_gs(\tv)=D_gs(V)$. Then $V-\tv \in Ker(D_gs)\cap Ker(\om|_g) = Ker(\mu)^R|_g$, showing that $dim(Ker(\om|_g))\leq rk(Ker(\mu))+rk(Ker(\mus))$. This finishes the proof that $Ker(\om)$ has constant rank.

The same argument using the broken path \eqref{eq:gpathid} can be also applied to show that $Ker(\om)\subset Ker(d\om)$, by recalling that the infinitesimal components of $d\om$ are $(\eta,0)$ and using the properties of $(\mu,\eta)$ appearing in Lemma \ref{lem:keromlin}.
\end{proof}

We summarize  the discussion in the following theorem.

\begin{theorem}\label{thm:Gomkerred}
	Let $G\rra M$ be a Lie groupoid with connected source fibers and $\om$ a multiplicative $k$-form on $G$. If $Lie(\om)$ is kernel-reducible in $Lie(G)$, then the distribution $Ker(\om)\subset TG$ is integrated by a regular foliation $\F_{Ker(\om)}$. In addition, if the quotient map $\q:G\to G/\F_{Ker(\om)}$ defines a simple quotient of $G\rra M$, then $\om$ is kernel reducible. In this case,
	$$ \text{$(\q_*G,\q_*\om)$ integrates $(\tilde A, (\tmu,\teta))$,}$$
	where $q:Lie(G) \to \tilde A$ is the simple quotient of $Lie(G)$ by $Ker(Lie(\om))$ and $(\tmu,\teta)$ are the components of $q_*Lie(\om)$.
\end{theorem}

\section{Quotients of geometric structures}\label{sec:Applications}

It is well known that Poisson and Dirac structures can be seen as a Lie algebroid equipped with an IM 2-form suitably non-degenerate \cite{bcwz}. A similar description can be done in the more general framework of higher Poisson and Dirac structures \cite{BMR}. Motivated by the study of quotients of Poisson and Dirac structures, and their higher versions, in this section we study quotients of Lie algebroids endowed with morphic forms for which the corresponding form on the quotient defines the geometry of interest. The original data can be thus seen as \emph{generalized quotient data} which does not need to consist of a structure of the same type. We also use the Lie theory for morphic forms developed in Section \ref{sec:mult} to give integration schemes for the resulting quotient structures. In these schemes, the integration is obtained as a quotient of a Lie groupoid endowed with a suitable basic multiplicative form. We start by studying the case of quotients of IM 2-forms yielding a Dirac structure, e.g. a Poisson quotient. The case of higher Dirac structure is treated separately at the end of this section.

\subsection{IM 2-forms, Poisson and Dirac structures}\label{subsec:IMvsDirac}

Let $A$ be a Lie algebroid structure on a vector bundle $E\to M$ and $\chi\in\Omega^3_{cl}(M)$ a closed 3-form. Let $\eta_{\chi}:E\to \wedge^2T^*M$ be the bundle map defined by 

\begin{equation}\label{eq:etachi}
\eta_{\chi}(e):=-i_{\rho(e)}\chi, 
\end{equation}
where $\rho:E\to TM$ is the anchor map of $A$.

\begin{definition}\label{def:IMtwistednondeg}
An IM 2-form $(\mu,\eta)$ on $A$ is called:

\begin{itemize} 

\item \textbf{$\chi$-twisted} if $\eta=\eta_{\chi}$ 

\item \textbf{non-degenerate} if $\mu:E\to T^*M$ is an isomorphism

\item \textbf{weakly non-degenerate} if $rk(E)=dim(M)$ and $Ker(\mu)\cap \Ker(\rho)=0.$

\end{itemize}
\end{definition}

It is clear from the definition that every non-degenerate IM 2-form is weakly non-degenerate, but the converse is not true in general. As shown in \cite{dHO}, thinking of the anchor map $\rho:E\to TM$ as a 2-term complex of vector bundles, an IM 2-form is non-degenerate (respectively weakly non-degenerate) if and only if it defines an isomorphism (respectively a quasi-isomorphism) of 2-term complexes.
It is easy to see that in the case of a $\chi$-twisted IM 2-form, equations \eqref{eq:IMeqs} become equivalent to

\begin{eqnarray}
		i_{\rho(e_1)}\mu(e_2)& = &- i_{\rho(e_2)}\mu(e_1) \nonumber\\
		\mu([e_1,e_2])& =& L_{\rho(e_1)}\mu(e_2)-i_{\rho(e_2)}d\mu(e_1) + i_{\rho(e_2)}i_{\rho(e_1)}\chi 
		\label{eq:twistedIM}
		\end{eqnarray}

There is a close relation between IM 2-forms and Dirac structures, which we proceed to recall now. For more details, see \cite{bcwz} and the references therein. A \textbf{$\chi$-twisted Dirac structure} on $M$ is a subbundle $L\subset TM\oplus T^*M$ which is Lagrangian (i.e. $L=L^{\perp}$) with respect to the symmetric pairing $\langle (X,\alpha), (Y,\beta)\rangle=\alpha(Y)+\beta(X)$ and involutive with respect to the twisted Courant bracket 

$$\Cour{(X,\alpha),(Y,\beta)}=([X,Y],L_X\beta-i_Yd\alpha+\chi(X,Y,\cdot)).$$

There is a one-to-one correspondence between $\chi$-twisted Dirac structures on $M$ and Lie algebroids on $E\to M$ together with a $\chi$-twisted IM 2-form which is weakly non-degenerate. The correspondence is given as follows: a $\chi$-twisted and weakly non-degenerate IM 2-form $(\mu,\eta)$ on $A$ determines a Dirac structure $L\subset TM\oplus T^*M$ defined by

\begin{equation}\label{eq:diracofIM}
L:=\{(\rho(e),\mu(e)); e\in E\}.
\end{equation}

Conversely, a $\chi$-twisted Dirac structure $L\subset TM\oplus T^*M$ can be seen as a Lie algebroid on $E=L\to M$ with anchor $L\to TM$ given by the canonical projection and Lie bracket defined by the Courant bracket. In this case, there is an induced $\chi$-twisted IM 2-form $(\mu,\eta_{\chi})$ where $\mu:L\to T^*M$ is given by the cotangent projection.

An important particular case is given by Dirac structures defined as the graph of a bivector. For that, recall that any bivector $\pi\in\Gamma(\wedge^2TM)$ defines a Lagrangian subbundle $L_{\pi}:=\{(\pi(\alpha,\cdot),\alpha);\alpha\in T^*M\}$ called the graph of $\pi$. If $L_{\pi}$ is involutive with respect to the $\chi$-twisted Courant bracket, then $\pi$ is called a \textbf{$\chi$-twisted Poisson structure}. In the case when $\chi=0$, the bivector $\pi$ is a usual Poisson structure on $M$.

The \textbf{kernel} of a Dirac structure $L$ is defined by $Ker(L)=L\cap TM$. It is well-known that $L$ is given by the graph of a bivector if and only if $Ker(L)=0$. Let $L$ be a $\chi$-twisted Dirac structure on $M$, so it is defined by an IM 2-form $(\mu,\eta_{\chi})$ as in \eqref{eq:diracofIM}. One has that $Ker(L)=\rho(Ker(\mu))$ and we conclude that $Ker(L)=0$ if and only if $\mu:E\to T^*M$ is an isomorphism. In other words, Poisson structures seen as Dirac structures correspond to non-degenerate IM 2-forms. The correspondence is given by

\begin{equation}\label{eq:pifrommu}
\pi^{\sharp}(\alpha)=\rho(\mu^{-1}(\alpha)), \ \alpha\in T^*M.
\end{equation}
Additionally, when $(\mu,\eta_{\chi})$ is $\chi$-twisted, then the bivector \eqref{eq:pifrommu} is a $\chi$-twisted Poisson structure.

\subsection{Twisted quotients of IM $2$-forms}\label{subsect:twistedquotientIM}

In this subsection we characterize basic IM $2$-forms whose quotients are twisted. The corresponding notion of non-degeneracy of a quotient IM 2-form will be treated in subsection \ref{subsec:applications}.

Let us consider $q:E\to \tE$ a simple quotient vector bundles. Let $A$ be a Lie algebroid structure on $E$ which is $q$-basic, hence it yields a quotient Lie algebroid structure $\tA$ on $\tE$. The next result is an immediate consequence of the characterization of basic linear forms of Proposition \ref{prop:linomqbasic}.

\begin{proposition}\label{prop:twistedquotientIM}

Let $\tchi\in\Omega^3_{cl}(\tM)$ be a closed 3-form and $(\mu,\eta)$ a $q$-basic IM $2$-form on $A$ with associated quotient IM 2-form $(\tmu,\teta)$ on $\tilde A$. Then $(\tmu,\teta)$ is $\tchi$-twisted if and only if $(\mu,\eta)$ is $\chi$-twisted with $\chi:=q^*_M\tchi$. In particular, $\teta=0$ iff $\eta=0$.
\end{proposition}

In Proposition \ref{prop:linomqbasic} we gave a characterization of basic linear $k$-forms in terms of their components. In the special case of $2$-forms, there is a simpler characterization which we proceed to explain now. For that, we consider a simple quotient $q:A\to \tA$ with components $(q_M:M\to \tM,K,\Delta)$. Suppose that $(\mu,\eta)$ is an IM 2-form on $A$ which is $q$-basic and induces $(\tmu,\teta)$ on $\tA$. Then, Proposition \ref{prop:linomqbasic} guarantees that the following conditions hold:

\begin{itemize}

\item[i)] $K\subseteq Ker(\mu)$ and $K\subseteq Ker(\eta)$
\item[ii)] for every invariant section $e\in \Gamma(E)$ the differential forms $\mu(e), \eta(e)$ are $q_M$-basic

\end{itemize}

As a consequence of i) and ii) above, there is a well defined bundle map 

$$\overline\mu:E/K\to Ann(Ker(Tq_M)); \ e\cdot mod(K)\mapsto \mu(e).$$ 
This map is equivariant with respect to the representations $\Delta$ and $\Delta^{\nu^*}$ of $M\times_{\F(q_M)} M$. See Example \ref{ex:conormalquot} for the definition of
the simple quotient $q_M^{\nu^*}:Ann(Ker(Tq_M))\to T^*\tM$ induced by $q_M$.

\begin{remark}\label{rmk:twistKeta} Notice that for a \emph{not necessarily basic} IM 2-form $(\mu,\eta)$ on $A$ which is twisted by $q^*_M\tchi$ for some $\tchi \in \Omega^3_{cl}(\tM)$, the conditions involving $\eta$ in i) and ii) above, hold automatically. This follows from the definition of $\eta=\eta_{q^*_M\tchi}$ combined with the fact that $\rho(k)\in \mathfrak{X}(M)$ is projectable to zero for every $k\in \Gamma(K)$.

\end{remark}

The next proposition gives a natural characterization of basic IM 2-forms whose quotient is a \emph{twisted} IM 2-form.

\begin{proposition}\label{lma:thequotistwist}
	Let $q:E\to \tE$ be a simple quotient with components $(q_M:M\to \tM,K,\Delta)$, $A$ a $q$-basic Lie algebroid structure on $E$ and $(\mu,\eta)$ a IM $2$-form on $A$ twisted by $\chi:=q^*_M\tchi$ for some $\tchi \in\Omega^3_{cl}(\tM)$. The following are equivalent:

\begin{enumerate}

\item $(\mu,\eta)$ is $q$-basic

\item $K\subset Ker(\mu)$, $Im(\mu)\subset Ann(Ker(Tq_M))$ and the induced map $\overline{\mu}:E/K\to  Ann(Ker(Tq_M))$ is equivariant. 

\end{enumerate}

\end{proposition}

\begin{proof}
We just need to show that \emph{2}. implies \emph{1}. In order to see that $(\mu,\eta)$ is basic we check the conditions as in Proposition \ref{prop:linomqbasic}. It follows from Remark \ref{rmk:twistKeta} that $K\subseteq Ker(\eta)$ and $\eta(e)$ is $q_M$-basic for every invariant section $e\in\Gamma(E)$. The equivariance of $\overline{\mu}:E/K\to  Ann(Ker(Tq_M))$ implies that $\mu(e)$ is $q_M$-basic for every invariant section $e\in\Gamma(E)$. 

\end{proof}


\subsubsection*{Kernel-reducible twisted IM $2$-forms}

We now investigate kernel-reducibility of morphic $2$-forms. As a motivation to have in mind, the quotient of a linear $2$-form by its kernel yields a linear $2$-form which is non-degenerate. Thus, when the form is also morphic, this process produces twisted bivectors as an outcome. 
These twisted quotient bivectors will be studied in detail in the following subsection.

We begin by showing that, in the case of linear $2$-forms, there is a simplification of the charaterization of the constant rank condition. We think that this result could be of independent interest.
For $2$-forms, $\mus$ coincides with the transpose of $\mu$,
$$ \mus = \mu^t: TM \to E^*.$$
The inclusion $Im(\mu)\subset Ann(Ker(\mu^\sharp))$ holds for any linear $k$-form. In the case $k=2$, using $dim(Im(\mu|_x))=dim(Im(\mu^t|_x))$, the inclusion becomes an equality by dimension count,
\begin{equation}\label{eq:Immux2}
Im(\mu|_x) = Ann(Ker(\mus|_x)), \ \forall x\in M.
\end{equation}

\begin{lemma}\label{lma:kerred2f}
	Let $\om$ be a linear $2$-form on $E\to M$ with components $(\mu,\eta)$ and assume $Ker(\om)\subset Ker(d\om)$. Then, $Ker(\om)$ has constant rank iff $Ker(\mu)$ has constant rank and $Ker(\mu^\sharp)\subset TM$ is integrable as a distribution.
\end{lemma}

\begin{proof}
	The implication $\Rightarrow$ follows directly from Lemma \ref{lem:keromlin}. For the converse, we first notice that since $\mu^\sharp=\mu^t$, then $\mu$ has constant rank if and only if $\mu^\sharp$ has constant rank. Also, by eq. \eqref{eq:Immux2}, condition (3.) in Lemma \ref{lem:keromlin} reads:
	$$ \text{(3'.): }i_v d \mu(e) \in Ann(Ker(\mu^\sharp)), \ \forall v\in \Gamma(Ker(\mu^\sharp)), \ e \in \Gamma(E),$$
	where we have used that $Ker(\mus)\subset Ker(\eta^\sharp)$ as a consequence of the hypothesis $Ker(\om)\subset Ker(d\om)$. The implication $Ker(\om)\subset Ker(d\om) \Rightarrow Ker(\mus)\subset Ker(\eta^\sharp)$ follows without assuming that $Ker(\om)$ has constant rank, see eq. \eqref{eq:Vkercoord} at $u=0$.
	We argue that (3'.) above automatically holds when $Ker(\mu^\sharp)$ is integrable. The key point is that, around each point in $M$, we can take an adapted chart $U=F\times \tU$ such that the distribution is given by the fibers of the projection $U \to \tU$, and thus $\Gamma_U Ann(Ker(\mus))$ is locally given by $C^\infty(U)$-linear combinations of exact $1$-forms $d\tilde x^k$, with $\tilde x^k$ coordinates in $\tU$. Since $\mu(e)\in \Gamma_U Ann(Ker(\mus))$ for any section $e$, one verifies that $i_{v_2}i_{v_1}d\mu(e) =0$ for $v_1,v_2 \in Ker(\mus)$, and thus $i_{v_1}d\mu(e) \in Ann(Ker(\mus))$ as wanted. We thus have proven (1.,2.,3.) of Lemma \ref{lem:keromlin}, implying that $Ker(\om)$ has constant rank.
\end{proof}

We now summarize the conditions characterizing when a kernel-reducible IM $2$-form gives rise to a twisted-closed form on the quotient. The following is a direct consequence Corollary \ref{cor:linomkerred} and the previous results of this subsection. First observe that if $\mu$ has constant rank and $Ker(\mus)=Ker(Tq_M)$ for a simple quotient $q_M:M\to \tM$, then by \eqref{eq:Immux2} the induced map
$$ \overline{\mu}: E/Ker(\mu) \to Ann(Ker(Tq_M)) \text{ is an isomorphism over $id_M$.}$$
This isomorphism can be used to define a representation $\Delta$ of $M\times_{\F(q_M)} M$ on $E/Ker(\mu)$ so that $\overline{\mu}$ is equivariant with respect to the conormal representation $\Delta^{\nu^*}$ on $Ann(Ker(Tq_M))$ (see Example \ref{ex:conormalquot}). Then $(q_M,Ker(\mu),\Delta)$ define the components of a simple quotient $q$ of $E\to M$, and $\mu$ descends to a map
$$ \tmu: q_*E \to T^*\tM$$
which is an isomorphism.

\begin{corollary}\label{cor:kerred2}
	Let $A$ be a Lie algebroid on $E\to M$ and $(\mu,\eta)$ be the components of a linear $2$-form on $E$. Then, $(\mu,\eta)$ defines a kernel-reducible IM $2$-form on $A$ whose quotient is twisted iff the following conditions hold:
	\begin{enumerate}
		\item $\mu$ has constant rank
		\item the regular distribution defined by $Ker(\mu^t)\subset TM$ is integrable and defines a simple quotient $q_M:M\to \tM$ onto its leaf space
		\item there exists a closed 3-form $\tchi\in \Omega^3_{cl}(\tM)$ such that $\eta=\eta_{q_M^*\tchi}$ 
	\end{enumerate}
	In such a case, the quotient component $\tmu:\tE\to T^*\tM$ is an isomorphism.
\end{corollary}

\subsection{Applications in Poisson and Dirac geometry}\label{subsec:applications}

In this subsection we apply our results on quotients of IM 2-forms to reduced spaces in Poisson and Dirac geometry. We also describe the corresponding integrations in the sense of \cite{bcwz}.

\subsubsection*{Integration of twisted quotient IM 2-forms}

Let $\chi\in\Omega^3(M)$ be a closed 3-form on a manifold $M$. A \textbf{twisted presymplectic groupoid} over $M$ \cite{bcwz} is a Lie groupoid $G\rra M$ with $dim(G)=2dim(M)$ together with a multiplicative 2-form $\omega\in\Omega^2(G)$ with $d\omega=s^*\chi-\tau^*\chi$ and satisfying the following non-degeneracy condition

\begin{equation}\label{eq:presympnondegenerate}
Ker(Ts(x))\cap Ker(T\tau(x))\cap Ker(\omega_x)=\{0\}; \ \text{for every } x\in M.
\end{equation}

In this case, as shown in \cite{bcwz}, the associated IM 2-form $(\mu,\eta)$ as in \eqref{eq:muetfrommult} is $\chi$-twisted and weakly non-degenerate as in Definition \ref{def:IMtwistednondeg}. In particular, there is an induced Dirac structure $L$ over M \eqref{eq:diracofIM}, which is isomorphic to the Lie algebroid $A$ of $G$ under the map $(\rho,\mu):A\to TM\oplus T^*M$. In the special case when $\omega$ is symplectic, so that $(G,\om)$ defines a \textbf{symplectic groupoid}, the IM 2-form $(\mu,\eta)$ is non-degenerate with $\eta=0$ and hence the induced Dirac structure on $M$ is given by the graph of a Poisson bivector.

Conversely, by \cite[Prop. 3.5 (iii)]{bcwz}, if $\om$ is a multiplicative $2$-form on a Lie groupoid $G\rra M$ such that its infinitesimal components $(\mu,\eta)$ define a $\chi$-twisted IM 2-form, then $d\omega=s^*\chi-\tau^*\chi$. Moreover, by \cite[Cor. 4.8]{bcwz}, if $(\mu,\eta)$ is weakly non-degenerate (resp. non-degenerate) then $(G,\om)$ defines a twisted presymplectic (resp. symplectic) groupoid. Given the infinitesimal data $(A,\mu,\eta)$, when $A$ is integrable, one can always integrate the IM form to a multiplicative form $\om$ on the source-simply-connected integration $G(A)\rra M$ of $A$ (see \cite{bursztyncabrera,bco,bcwz}).

Finally, we study the Lie theory of quotients of morphic $2$-forms. For that, we consider a simple quotient of Lie groupoids $\q:G\to \tG$ inducing a simple quotient of Lie algebroids $q:A\to\tA$. We assume that $G$ is a Lie groupoid with connected source fibers. Let $\omega\in\Omega^2(G)$ be a multiplicative 2-form integrating an IM 2-form $(\mu,\eta)$ on $A$. It follows from Proposition \ref{prop:twistedquotientIM} that, if $(\mu,\eta)$ is $q$-basic inducing a quotient IM 2-form $(\tmu,\teta)$ which is $\tchi$-twisted, then $(\mu,\eta)$ is $\chi$-twisted where $\chi=q^*_M\tchi$. Since $(\mu,\eta)$ integrates to $\omega$ then 

$$d\omega=s^*(q^*_M\tchi)-\tau^*(q^*_M\tchi).$$

Since $(\mu,\eta)$ is $q$-basic and $G$ is source-connected, by Prop. \ref{prop:omqbasicfrominfi}, the multiplicative 2-form $\omega\in\Omega^2(G)$ is basic with respect to $\q:G\to \tG$. Hence $\omega=\q^*\tom$ for a unique multiplicative 2-form $\tom\in\Omega^2(\tG)$ which is an integration of the $\tchi$-twisted IM 2-form $(\tmu,\teta)$.  We observe that
$$d\tom=s^*\tchi-\tau^*\tchi.$$

In this case, if $(\tmu,\teta)$ is weakly non-degenerate, then $(\tG,\tom)$ is a presymplectic groupoid over $M$ integrating the Dirac structure $\tilde L$ given by the image of $(\tilde \rho,\tmu):\tA\to T\tilde M \oplus T^*\tilde M$. In particular, if $(\tmu,\teta)$ is non-degenerate, then $(\tG,\tom)$ is a twisted \emph{symplectic} groupoid integrating the twisted Poisson bivector $\tilde{\pi}^{\sharp}:=\tilde\rho\circ \tilde{\mu}^{-1}$. An instance of the previous situation is presented in the next example.

\begin{example}\label{ex:integrationkernelreducible}
Let $(\mu,\eta)$ be a kernel-reducible IM 2-form with associated simple quotient $q=q_{Ker(Lie(\om))}$. In this case $\tmu$ is an isomorphism and hence induces a $\tchi$-twisted Poisson structure $\tpi$ on $\tM$. Further, if $\om$ is kernel-reducible inducing a simple quotient $\q:G\to\tG$, then $(\tG, \q_*\om)$  defines a twisted symplectic groupoid integrating the twisted Poisson manifold $(\tM,\tpi,\tchi)$. The kernel reducibility condition for $(G,\om)$ holds amounts to, when $G$ is source-connected so that the distribution $Ker(\om)\subset TG$ integrates to a regular foliation $\F\subset G$ by Thm. \ref{thm:Gomkerred}, the quotient map $\q:G\to G/\F$ defining a simple quotient of $G\rra M$.
\end{example}	


\subsubsection*{Quotient Dirac structures}

In what follows we study quotient IM 2-forms which give rise to a twisted Dirac structure.

For that, we consider a simple quotient of Lie algebroids $q:A\to \tA$ with components $(q_M:M\to \tM,K,\Delta)$ and a $q$-basic IM 2-form $(\mu,\eta)$ on $A$. It follows from Proposition  \ref{prop:twistedquotientIM} that $(\tmu,\teta)$ is $\tchi$-twisted for some $\tchi\in\Omega^3_{cl}(\tilde M)$ if, and only if, $(\mu,\eta)$ is $\chi$-twisted with $\chi= q^*_M\tchi$. Regarding non-degeneracy properties, the IM 2-form $(\tmu,\teta)$ is weakly non-degenerate if and only if the following holds:

	\begin{equation}\label{eq:quotientwnd}
			rk(K)=rk(E)-dim(\tM) \text{ and } K=Ker(\mu)\cap \rho^{-1}(Ker(Tq_M)).
	\end{equation}

In particular, $(\tmu,\teta)$ is non-degenerate if and only if 

\begin{equation}\label{eq:quotientnd}
rk(K)=rk(E)-dim(\tM) \text{ and } Ker(\mu)=K.
\end{equation}

The latter is equivalent to $Ker(\mu)=K$ and $rk(Im(\mu))=dim(\tM)$, so $Im(\mu)=Ann(Ker(Tq_M))$.\\

From now on we use the following terminology.

\begin{definition}\label{def:diracquotdata}
	Consider a set of data $(A,\mu,\eta,q)$ where $A$ is a Lie algebroid on a vector bundle $E\to M$, $q$ is a simple quotient of $A$, and $(\mu,\eta)$ defines a twisted IM $2$-form on $A$. We say that this data defines {\bf Dirac quotient data} when the IM $2$-form is $q$-basic and its quotient defines a weakly non-degenerate twisted IM $2$-form. 
\end{definition}

It follows from the results of subsection \ref{subsect:twistedquotientIM}, that a Dirac quotient data is equivalent to

\begin{corollary}\label{cor:diracquot}
	Let $A$ be a Lie algebroid on $E\to M$, $(\mu,\eta)$ be the components of a linear $2$-form on $E$ and $q$ a simple quotient of $A$ with components $(q_M:M\to \tM,K,\Delta)$. Then, $(A,\mu,\eta,q)$ defines Dirac quotient data iff
	\begin{enumerate}
		\item there exists a closed $3$-form $\tchi$ on $\tM$ such that $\eta=\eta_{q_M^*\tchi}$ as in eq. \eqref{eq:etachi},
		\item $K=Ker(\mu)\cap \rho^{-1}(Ker(Tq_M))$ and $rk(K)=rk(E)-dim(\tM)$ 
		\item $Im(\mu)\subset Ann(Ker(Tq_M))$ and the induced map $\overline{\mu}:E/K \to  Ann(Ker(Tq_M))$ is equivariant for the representations $\Delta$ and $\Delta^{\nu^*}$ of $M\times_{\F(q_M)} M$ (recall $\Delta^{\nu^*}$ from Example \ref{ex:conormalquot})
	\end{enumerate}
	
	\end{corollary}

If $(A,\mu,\eta,q)$ defines Dirac quotient data, then
	$$ \tilde L:= \{ (Tq_M(\rho(e)),\tilde \alpha )\in T\tM\oplus T^*\tM: q_M^*\tilde \alpha = \mu(e), \ e\in E \}$$
	given as in \eqref{eq:diracofIM} is a $\tchi$-twisted Dirac structure on $\tM$ and $q_*A \simeq \tilde L$ through $q_*(\mu,\eta)$.

	Regarding integration, if $(G\rra M, \om)$ integrates $(A,(\mu,\eta))$, $G$ is source-connected and there exists a simple quotient $\q:G\to \tG$ such that $Lie(\q)=q$, then $\om$ is $\q$-basic (Prop. \ref{prop:omqbasicfrominfi}) and $(\tG,\q_*\om)$ defines a $\tchi$-twisted presymplectic groupoid \cite{bcwz} integrating the quotient twisted Dirac manifold $(\tM,\tilde L,\tchi)$. (We observe that the source-connectedness is only used to ensure $\om$ is $\q$-basic, but the non-degeneracy condition on $\om$ is automatic, see \cite{bcwz}.)

\begin{remark}\label{rmk:poissonquotdata}
A Dirac quotient data $(A,\mu,\eta,q)$ with $\eta=0$ and for which the quotient IM 2-form $(\tmu,0)$ is \emph{non-degenerate} is referred to as a \textbf{Poisson quotient data}. In this case, the Lie groupoid $(\tG,\q_*\om)$ is a \emph{symplectic} groupoid integrating the Poisson manifold $(\tM,\tilde\pi)$ where $(\tilde\pi)^{\sharp}:=\trho\circ (\tmu)^{-1}$. Notice that when the linear form $\om$ defined by $(\mu,\eta)$ is kernel-reducible, together with the corresponding simple quotient $q=q_{Ker(\om)}$, they define Poisson quotient data since $\tmu$ is necessarily non-degenerate.
\end{remark}

The remaining of this subsection is devoted to discuss examples of both Poisson and Dirac quotients and their integrations.

\begin{example}[Regular Dirac structures as Poisson quotient data]

Here we consider a Dirac structure $L$ on $M$ with zero twist. We say that $L$ is \textbf{regular} if $Ker(L)=L\cap TM$ has constant rank. In this case, $Ker(L)$ defines a regular integrable distribution. A function $f\in C^\infty(M)$ is called {\bf admissible} if there exists a vector field $X_f$ such that $(X_f,df)\in \Gamma(L)$. Admissible functions inherit a Poisson bracket defined by  $\{f,g\}=dg(X_f)$. If $Ker(L)$ defines a simple foliation, then functions on the corresponding leaf space $\tM$ identify with admissible functions. In particular, the leaf space $\tM$ has an induced Poisson structure which can be shown to be defined by Poisson quotient data as defined in Remark \ref{rmk:poissonquotdata} above. Indeed, if $A$ is a Lie algebroid equipped with an IM 2-form $(\mu,0)$ inducing $L$ as in \eqref{eq:diracofIM}, then the condition of $L$ being regular is equivalent to $\mu$ having constant rank. Since $Ker(L)=Ker(\mu^t)$ is tangent to a simple foliation, the IM 2-form $(\mu,0)$ is kernel-reducible (see Lemma \ref{lma:kerred2f}) and hence it defines Poisson quotient data. The corresponding quotient Poisson structure $\tpi$ on $\tM$ coincides with the Poisson structure defined on admissible functions of $L$. Regarding the integration of $\tilde\pi$, if $(G,\omega)$ is a presymplectic groupoid with connected source fibers integrating $L$, then an integration of $(\tM,\tilde\pi)$ is given as in Example \ref{ex:integrationkernelreducible}. The integration of the Poisson structure on admissible functions was also described in \cite{JO}.  We observe that this integration strategy is a similar but alternative version of the more specific mechanism behind \cite[Thm. 1.1]{Alv}: since, in these cases, $q$ is of pullback type, the hypothesis of $q$ admitting an integration by a simple groupoid quotient is replaced by the criterion recalled in Example \ref{ex:intqpullbacktype} and one can show that, under this hypothesis, source-connectedness is not necessary to ensure $\om$ is $\q$-basic. ($(G,\om)$ is called a "$q_M$-admissible presymplectic integration" in \cite{Alv}.) This refined criterion can be used proceduce to integrations of a variety of Poisson manifolds obtained as quotients, including Poisson homogeneous spaces, see \cite{Alv,BIL}. We stress that their criterion to produce the groupoid simple quotient $\q$ integrating the pullback type quotient $q$ is non-trivial (see Example \ref{ex:intqpullbacktype} and the cited references) and is independent of the topics covered in the present paper.

\end{example}

\begin{example}[Simple Dirac quotients]

Let $L\subset TM\oplus T^*M$ be a zero-twisted Dirac structure and $q_M:M\to \tilde{M}$ a surjective submersion. It was shown in \cite{frejlich-marcut} that if $Ker(Tq_M)\oplus 0\subset Ker(L)$, then there is a unique Dirac structure $\tilde{L}$ on $\tilde{M}$ such that $q_M$ defines a (forward) Dirac map. We explain now how this result can be seen as an instance of a Dirac quotient data. For that we assume that $q_M$ has connected fibers, thus defining a simple quotient of manifolds. Viewing $L$ as a Lie algebroid with the closed IM 2-form $\mu:L\to T^*M$ defined as the cotangent projection, we have that 
	$$Ker(\mu)=\{ (X,0)\in L:X \in TM \} \simeq Ker(L).$$
Since $Ker(Tq_M)\subseteq Ker(L)$, the vector bundle $K:=Ker(Tq_M)\oplus 0 \subset L$. Notice that $L/K\subseteq TM/Ker(Tq_M) \oplus Ann(Ker(Tq_M))$ inherits a representation $\Delta$ of $M\times_{\F(q_M)} M$ given by the restriction of $\Delta^{\nu }\oplus \Delta^{\nu^*}$ (see Examples \ref{ex:normalquot} and \ref{ex:conormalquot}). It is straightforward to verify that $(K,\Delta)$ are the components of a simple quotient $q$ of the Lie algebroid $L\to M$. Notice also that

$$rk(K)=dim(M)-dim(\tM)=rk(L)-dim(\tM)$$
	and the hypothesis $\ker(Tq_M)\oplus 0\subset Ker(L)$ implies $K =Ker(\mu)\cap \rho^{-1}(Ker(Tq_M))$. Then by Corollary \ref{cor:diracquot} the resulting quotient Lie algebroid $\tilde L$ is a Dirac structure on $\tM$ which actually coincides with the one defined in \cite{frejlich-marcut} above. The fact that $q_M$ defines a Dirac map follows directly from the definition of $\tilde L$. The integration of $\tilde L$ can be handled along the general lines below Corollary \ref{cor:diracquot}.
\end{example}

\begin{example}[Simple Poisson quotients and Libermann's theorem] 
Let $q_M: (M,\pi)\to (\tM,\tpi)$ be a Poisson map such that $q_M$ is a simple quotient. Consider $E:=Ann(Ker(Tq_M)) \subset T^*M$ the vector subbundle over $M$ and $\mu:E \to T^*M$ given by the natural inclusion map. Observe that $(\mu,\eta=0)$ is obtained by "restriction" (see Remark \ref{rmk:restrictionIM}) of the canonical symplectic structure $\om_{can}\in \Omega^2(T^*M)$ seen as a linear $2$ form. Since $q_M$ is a Poisson map, then $E$ defines a Lie subalgebroid of $\hat A:=T^*_\pi M$ which we denote $A$. Since $\om_{can}$ is $T_\pi^*M$-morphic, then $\mu$ defines a closed IM $2$-form on $A$. By its definition, $\mu$ is injective ($Ker(\mu)=0$) and hence has constant rank. Also, since $Ker(\mu^t)=Ker(Tq_M)$, there is an induced simple quotient $q:E \to T^*\tM$ such that $(A,\mu,\eta=0,q)$ define Poisson quotient data. We thus get a unique Poisson structure $\tpi'$ on $\tilde M$ which, by the fact that $q_M$ is a Poisson map, it must yield $\tpi'=\tpi$. We have then encoded the "simple Poisson quotient" map $q_M$ as a particular case of Poisson quotient data. Conversely, 
if $(A,\mu,\eta=0,q)$ define Poisson quotient data with $Ker(\mu)=0$ and $Ker(\mus)=Ker(Tq_M)$, then the underlying $q_M$ defines a Poisson map, $E\simeq Ann(Ker(Tq_M))$ and both $A$ and $\mu$ are isomorphic to the structures defined above. Integration of $(\tM,\tpi)$ can be handled along the general lines below Corollary \ref{cor:diracquot}. When $\pi=\omega_M^{-1}$ is symplectic, we can recover Libermann's theorem: if the $\omega_M$-orthogonal $(\ker(Tq_M))^{\omega_M} \subseteq TM$ is involutive, then there exists a unique Poisson structure $\tilde \pi$ on $\tilde M$ such that $q_M:(M,\omega)\to (\tilde M,\tilde \pi)$ defines Poisson morphism. In this case, we can take $E:=(\ker(Tq_M))^{\omega_M} \subset TM$, which defines a Lie subalgebroid $A\hookrightarrow TM$, and consider the closed IM 2-form $\mu(v):=\omega_M(v,-)$. It follows that $(\mu,0)$ is kernel-reducible, with induced simple quotient $q:E\to T^*\tM$ and with $Ker(\mus)=Ker(Tq_M)$, so that $(A,\mu,\eta=0,q)$ defines Poisson quotient data inducing $\tilde \pi$ on $\tM$ as in Libermann's theorem. 
\end{example}

\begin{example}[Poisson reduction from graded presymplectic supergeometry] 
Moti-vated by a supergeometric version of presymplectic reduction, Cattaneo and Zambon introduced in \cite{CZ12} a reduction scheme yielding Poisson structures as outcomes. 
The involved data can be understood as a particular case of our Poisson quotient data, as we now proceed to explain. The input data consists of a vector sub-bundle $(E\to M)\subset (T^*\hat M\to \hat M)$ and a Poisson structure $\hat \pi$ on $\hat M$. Define a linear $2$-form $\om$ on $E$ (with components $(\mu,0)$) by pullback of the canonical symplectic form $\om_{can}$ along the inclusion map $I: E\hookrightarrow T^*\hat M$. Using Lemma \ref{lem:keromlin} and denoting $Ann_{\hat M}(E)\subset T\hat M|_M$ the annihilator space of $E\subset T^*\hat M|_M$, one can verify directly that $\om=I^*\om_{can}$ has constant rank iff $T\F_M:=TM\cap Ann_{\hat M}(E)\subset TM$ defines a constant rank involutive distribution on $M$ (c.f. \cite[Prop. 5.4]{CZ12}). On the other hand, $I$ defines a Lie sub-algebroid   of $\hat A=T^*_{\hat \pi} \hat M$ iff (c.f. \cite[hypothesis of Prop. 5.15]{CZ12}) 
$$ \hat \pi^\sharp (E) \subset TM \text{ and } d \hat \pi(d\hat f_1,d\hat f_2)|_x \in Ann_{\hat M}(T\F_M) \text{ whenever $d\hat f_j|_x \in E$ and $x\in M$.} $$
We denote $A$ the induced Lie algebroid structure on $E\to M$. Assume further that $\om$ is kernel-reducible, which is equivalent to $q_M: M \to M/\F_M$ defining a simple quotient by Corollary \ref{cor:kerred2} (c.f. \cite[Prop. 5.11]{CZ12}). In this case, $(A, \mu,\eta= 0, q\equiv q_{Ker(\om)})$ define Poisson quotient data as in Remark \ref{rmk:poissonquotdata}. We thus obtain we obtain a \emph{reduced Poisson structure} $\tpi$ on $\tM=M/\F_M$ and recover the reduction procedure of \cite[Prop. 5.15]{CZ12}.
We also notice that this example yields the infinitesimal counterpart to the following situation: if a Lie groupoid $G\rra M$ is endowed with a closed multiplicative $2$-form $\om$ with constant rank such that $Ker(\om)$ defines a simple quotient of Lie groupoids $\q:G \to \tG$, then $(\tG\rra \tM, \q_*\om)$ defines a symplectic groupoid and, hence, there is an induced Poisson structure $\tpi$ on $\tM$ characterized by
$$ \tpi^\sharp(\tilde \alpha) = Tq_M(\rho(e)) \in T\tM, \text{ whenever } \mu(e) = q_M^*\tilde \alpha, \ \tilde \alpha \in T^*\tM, \ e\in E_G.$$
A particular case arises when $G\hookrightarrow (\hat G,\hat \om)$ defines an (immersed) coisotropic groupoid of a symplectic groupoid, leading infinitesimally to $E_G\subset (T^*\hat M, \omega_{can})$ defining a coisotropic submanifold (in the usual sense, see details in \cite[Prop. 4.2]{CZ12}). Coming back to the general infinitesimal case, integration of $(\tM,\tpi)$ can be handled as follows: if $(G\rra M,\om)$ integrates $(A,(\mu,0))$, $G$ is source-connected and there is a simple quotient $\q:G\to \tG$ of $G\rra M$ integrating the simple quotient $q:A\to \tA$ defined by the kernel of $(\mu,0)$, then $\om$ is $\q$-basic and $(\tG,\q_*\om)$ is a symplectic groupoid integrating $(\tM,\tpi)$. 
Finally, we mention two particular cases: Poisson and regular coisotropic submanifolds. When $M\hookrightarrow (\hat M,\hat \pi)$ defines a Poisson submanifold, then $E:= T^*\hat M|_{M}$ induces  Poisson quotient data as above such that $q_*A \simeq T^*_\pi M$. A characterizing property in this case is that the underlying component $\mu$ of $\om$ is fiber-wise surjective. Next, assume that $M\hookrightarrow (\hat M,\hat \pi)$ defines a coisotropic submanifold which is \emph{regular} in the sense that 
$$D:= \hat \pi^\sharp(Ann_{\hat M}(TM))\subset TM \text{ has constant rank,}$$
where $Ann_{\hat M}$ denotes the annihilator space in the ambient $\hat M$. If the foliation defined by $D$ is simple with underlying simple quotient $q_M:M\to \tM$, then $E:= Ann_{\hat M}(D)$ induces Poisson quotient data as above such that $q_*A\simeq T^*_{\tilde \pi} \tM$ where $\tilde \pi$ is the Poisson structure obtained from $\hat \pi$ via standard coisotropic reduction. (In the particular case in which $\hat \pi^\sharp|_{Ann_{\hat M}(TM)}$ is injective, the underlying simple quotient is of pullback type and a specialized integration treatment can be found in \cite[Prop. 3.17]{Alv}.) This particular case shows that standard coisotropic reduction can be handled with a particular instance of our Poisson quotient data.
\end{example}

\subsection{Applications to higher geometric structures}

In this final subsection, we illustrate some applications to quotients of higher geometric structures involving morphic $(k+1)$-forms with $k\geq 1$. We will focus on the notion of higher Dirac structure as defined  in \cite{BMR}, which contains the higher Poisson structures of \cite{burcabigl}. These structures can be regarded as consisting of a Lie algebroid together with a morphic $(k+1)$-form, just as ordinary Dirac and Poisson can with $k=1$, and we thus study quotient data adapted to them. This subsection thus provides a higher analogue of the previous subsections. For simplicity, we restrict ourselves to untwisted structures as they appear in the literature, although the incorporation of a twist into the discussion would be straightforward.

We begin by recalling the definitions of higher Dirac and Poisson structures. To that end, let $M$ be a manifold and consider the higher standard Courant algebroid structure on the vector bundle
$$ TM \oplus \Lambda^k T^*M \to M$$
consisting of the symmetric $\Lambda^{k-1}T^*M$-valued pairing $\langle (X,\alpha),(Y,\beta)\rangle = i_Y \alpha + i_X \beta$, the projection onto $TM$ as the anchor map, and the Courant-Dorfman bracket on sections given by
$$ \Cour{(X,\alpha),(Y,\beta)}= ([X,Y],L_X \beta - i_Yd\alpha). $$
A {\bf higher Dirac structure} (\cite{BMR}) is a subbundle $L\subset  TM \oplus \Lambda^k T^*M$ over $M$ such that its sections are closed under the Courant-Dorfman bracket and $L$ is \textbf{weakly lagrangian}, i.e. denoting $L^\perp$ the orthogonal relative to $\langle,\rangle$,
\begin{equation}\label{eq:weaklag}
\text{$L\subset L^\perp$, and $L\cap TM = L^\perp \cap TM$.}
\end{equation}
In the particular case in which $L^\perp \cap TM=0$, we say that $L$ defines a {\bf higher Poisson structure} (or "$k$-Poisson structure" in \cite{burcabigl}). Ordinary Dirac and Poisson structures correspond to the case $k=1$, see \cite{BMR}.

In analogy with ordinary Dirac structures, a higher Dirac structure is equivalent (see \cite[\S 5]{BMR}) to a pair $(A,\mu)$ where $A$ is a Lie algebroid over $M$ and $\mu$ defines a closed IM $(k+1)$-form satisfying the weakly non-degeneracy condition: 
\begin{equation}\label{eq:higherweaknond}
 \text{$Ker(\mus) = \rho(Ker(\mu))$ and $Ker(\mu)\cap Ker(\rho) = 0$.}
\end{equation}
Given the second, the first condition above results equivalent to $rk(E)=dim(M)$ for $k=1$, recovering Definition \ref{def:IMtwistednondeg}. A higher Dirac structure $L$ on $M$ determines a natural Lie algebroid structure $A$ on $E=L\to M$ and $\mu:=pr_{\Lambda^kT^*M}|_L:L \to \Lambda^kT^*M$ yields the associated weakly non-degenerate closed IM $(k+1)$-form. For a higher Poisson structure, the equivalent pair $(A,\mu)$ must satisfy the stronger non-degeneracy condition 
\begin{equation}\label{eq:highernondeg}
\text{$Ker(\mu)=0$ and $Ker(\mus)=0$ (see \cite[Prop. 4.1]{burcabigl}).}
\end{equation}
Notice that this non-degeneracy condition deviates for $k>1$ from the condition saying $\mu$ is an isomorphism (the standard case $k=1$). Observe also that the linear form $\om\in \Omega^{k+1}(E)$ defined by a higher Poisson structure is multisymplectic (see e.g. \cite{burcabigl}): closed and $i_V\om =0 \Rightarrow V=0$ (see eq. \eqref{eq:Vkercoord}).

We are now ready to introduce higher versions of Definition \ref{def:diracquotdata}.
\begin{definition}
	Consider the data $(A,\mu,q)$ where $A$ is a Lie algebroid on $E\to M$, $\mu:E\to \Lambda^kT^*M$ defines a closed IM $(k+1)$-form on $A$ and $q$ defines a simple quotient of $A$. We say that this data defines {\bf higher Dirac (resp. Poisson) quotient data} when the IM $(k+1)$-form is $q$-basic and its quotient satisfies condition \eqref{eq:higherweaknond} (resp. \eqref{eq:highernondeg}). 
\end{definition}
Denoting $\tmu$ the quotient closed IM $(k+1)$-form on $\tA=q_*A$, we get that $(\tA,\tmu)$ is equivalent to a higher Dirac structure or, respectively, a higher Poisson structure. Observe that, in the case of quotient data given by the kernel of a kernel-reducible form, the non-degeneracy condition \eqref{eq:highernondeg} is automatically satisfied by $\tmu$ (since the kernel was modded out) leading naturally to a higher Poisson quotient.

The characterization of higher quotient data in terms of components can be obtained from the general results of Section \ref{sec:vbalg}. The integration of the resulting higher structures on the quotient can be handled by our general Lie theoretic procedures of Section \ref{sec:mult}. We summarize the discussion in the following:
\begin{corollary}\label{cor:higherdiracquot}
	Let $A$ be a Lie algebroid on $E\to M$, $\mu:E\to \Lambda^kT^*M$ be the component of a closed linear $(k+1)$-form on $E$ and $q$ a simple quotient of $A$ with components $(q_M:M\to \tM,K,\Delta)$. Then, $(A,\mu,q)$ defines higher Dirac (resp. Poisson) quotient data iff 
	\begin{enumerate}
		\item $\mu$ satisfies the IM equations \eqref{eq:IMeqs} with $\eta=0$,
		\item $K=Ker(\mu)\cap \rho^{-1}(Ker(Tq_M))$ and $Ker(\mus)=\rho(Ker(\mu))+Ker(Tq_M)$ (resp. $K=Ker(\mu)$ and $Ker(\mus)=Ker(Tq_M)$) 
		\item for every $q$-invariant section $e\in \Gamma E$, $\mu(e)\in \Omega^{k}(M)$ is $q_M$-basic.
	\end{enumerate}
	
	When $(A,\mu,q)$ defines higher Dirac (resp. Poisson) quotient data, then 
	$$ \tilde L:= \{ (Tq_M(\rho(e)),\tilde \alpha )\in T\tM\oplus \Lambda^k T^*\tM: q_M^*\tilde \alpha = \mu(e) , e\in E\}$$
	defines a higher Dirac (resp. Poisson) structure on $\tM$ which is equivalent to $(q_*A,q_*(\mu,0))$ as outlined above.
	
	If $(G\rra M, \om)$ integrates $(A,(\mu,0))$, $G$ is source-connected and there exists a simple quotient $\q:G\to \tG$ of $G\rra M$ such that $Lie(\q)=q$, then $\om$ is $\q$-basic and $(\tG,\q_*\om)$ defines a``$k$-presymplectic groupoid" (see \cite[Def. 5.2]{BMR}) (resp. ``multisymplectic groupoid", see \cite{burcabigl}) integrating the higher Dirac (resp. Poisson) manifold $(\tM,\tilde L)$.
\end{corollary}
The source-connectedness condition is used to ensure $\om$ is $\q$-basic, but it is not needed for $\q_*\om$ to be ``weakly non-degenerate", since $Lie(\q_*\om)$ is of higher Dirac type (see \cite[Lemma 5.1]{BMR}).




\end{document}